\documentclass[final,leqno,onefignum,onetabnum]{siamltex1213}
\usepackage{amsmath}
\usepackage{enumerate}
\usepackage{amsfonts}
\usepackage{graphicx}

\newtheorem{fact}[theorem]{Fact}
\newtheorem{MKTremark}[theorem]{Remark}
\newtheorem{MKTexample}[theorem]{Example}
\usepackage{xparse}
\DeclareDocumentEnvironment{remark}{o}{\IfValueTF{#1}{\begin{MKTremark}[#1]}{\begin{MKTremark}\unskip}\rm}{\qquad$\diamond$\end{MKTremark}}
\DeclareDocumentEnvironment{example}{o}{\IfValueTF{#1}{\begin{MKTexample}[#1]}{\begin{MKTexample}\unskip}\rm}{\qquad$\diamond$\end{MKTexample}}
\usepackage[boxruled]{algorithm2e}

\newcommand{\RR}{\mathbb{R}}
\newcommand{\NN}{\mathbb{N}}
\newcommand{\di}{{\rm dist}}
\DeclareMathOperator*{\argmin}{arg\,min}

\title{Convergence rate analysis for averaged fixed point iterations in the presence of H\"older regularity}
\author{Jonathan M. Borwein\thanks{CARMA, U. Newcastle,
                                  Callaghan NSW 2308, Australia
                                  (\email{jonathan.borwein@newcastle.edu.au}).}
        \and
        Guoyin Li\thanks{Dept. Applied Math., U. New South Wales,
                         Sydney NSW 2052, Australia
                         (\email{g.li@unsw.edu.au}).}
          \and
          Matthew K. Tam\thanks{Institut f\"ur Numerische und Angewandte Mathematik,
							   U. G\"ottingen, 37083 G\"ottingen, Germany
                               (\email{m.tam@math.uni-goettingen.de}).}
          }

\date{\today}

\begin{document}
\maketitle
\slugger{siopt}{xxxx}{xx}{x}{x--x}

\begin{abstract}
 In this paper, we establish sublinear and linear convergence of fixed point iterations generated by averaged operators in a Hilbert space. Our results are achieved under a bounded H\"older regularity assumption which generalizes the well-known notion of bounded linear regularity. As an application of our results, we provide a convergence rate analysis for
 many important iterative methods in solving
 broad mathematical problems such as  convex feasibility problems and variational inequality problems. These include Krasnoselskii--Mann iterations, the cyclic projection algorithm, forward-backward splitting and the Douglas--Rachford feasibility algorithm along with some variants. In the important case in which the underlying sets are convex sets described by convex polynomials in a finite dimensional space, we show that the H\"older regularity properties are automatically satisfied, from which sublinear convergence follows.
\end{abstract}



\pagestyle{myheadings}
\thispagestyle{plain}
\markboth{\MakeUppercase{J.M. Borwein, G. Li, and M.K. Tam}}
         {\MakeUppercase{Convergence rate analysis of fixed point iterations}}

\section{Introduction}
 Consider the problem of finding a point in the intersection of a finite family of closed convex subsets of a Hilbert space; a problem often referred to as the \emph{convex feasibility problem} which arises frequently throughout areas of mathematics,  science and engineering. For details, we refer the reader to the surveys \cite{BB,censor2014}, the monographs \cite{BC2011,escalante}, any of \cite{combettes,Jon_AZNIAM,entropy}, and the references therein.

 One approach to solving convex feasibility problems involves designing a nonexpansive operator whose fixed point set can be used to easily produce a point in the target intersection (in the simplest case, the fixed point set coincides with the target intersection). The operator's fixed point iteration can then be used as the basis of an iterative algorithm which, in the limit, yields the desired solution. An important class of such methods comprises the so-called \emph{projection and reflection methods} which employ various combinations of \emph{projection} and \emph{reflection} operations with respect to underlying constraint sets.

 Notable methods of this kind include the \emph{alternating projection algorithm}\cite{Heinz_Set,GPR,Bor_Li_Yao}, the \emph{Douglas--Rachford (DR) algorithm} \cite{PDE,Lions_Mercier,Eckstein_Bertsekas}, and  many extensions and variants \cite{cycDRinfeas,BorTam,BNP,reich2015modular}. Even in settings without convexity \cite{Jon_AZNIAM,Jon_JOTA,Attouch1,Bor_Sims,Lewis,MOR}, such methods remain a popular choice due largely to their simplicity, ease-of-implementation and relatively -- often surprisingly -- good performance.

 The origins of the Douglas--Rachford (DR) algorithm can be traced to \cite{PDE} where it was used to solve problems arising in nonlinear heat flow. In its full generality, the method finds zeros of the sum of two maximal monotone operators. Weak convergence of the scheme was originally proven by Lions and Mercier \cite{Lions_Mercier}, and the result was recently strengthened by Svaiter \cite{svaiter}. Specialized to feasibility problems, Svaiter's result implies that the iterates generated by the DR algorithm are always weakly convergent, and that the \emph{shadow sequence} converges weakly to a point in the intersection of the two closed convex sets. The scheme has also been examined in \cite{Eckstein_Bertsekas} where its relationship with another popular method, the \emph{proximal point algorithm}, was discussed.

 Motivated by the computational observation that the Douglas--Rachford algorithm  sometimes outperforms other projection methods, in the convex case many researchers have studied the actual convergence rate of the algorithm. By \emph{convergence rate}, we mean how \emph{fast} the sequences generated by the algorithm converges to their limit points. For the Douglas--Rachford algorithm, the first such result, which appeared in \cite{Luke2} and was later extended by \cite{BBNP}, showed the algorithm  to converge linearly whenever the two constraint sets are closed subspaces with a closed sum, and, further, that the rate is governed exactly by the cosine of the \emph{Friedrichs angle} between the subspaces. In finite dimensions, if the sum of the two subspaces is not closed, convergence of the method -- while still assured -- need not be linear \cite[Sec.~6]{BBNP}. See also \cite{Giselsson} for other recent work regarding linear convergence. For most projection methods, it is typical that there exists instances in which the rate of convergence is arbitrarily slow and not even sublinear or arithmetic \cite{davis2015,HDH}. Most recently, a preprint of Davis and Yin shows that indeed the Douglas--Rachford method also may converge arbitrarily slowly in infinite dimensions \cite[Th.~9]{davis_yin}.

 In potentially nonconvex settings, a number of recent works \cite{Luke1,Luke2,Phan,Li_Pong} have established local linear convergence rates for the DR algorithm using commonly used constraint qualifications. When specialized to the convex case, these results state that the DR algorithm exhibits locally linear convergence for convex feasibility problems in a finite dimensional space whenever the relative interiors of the two convex sets have a non-empty intersection. On the other hand, when such a regularity condition is not satisfied, the DR algorithm can fail to exhibit linear convergence, even in simple two dimensional cases as observed by \cite[Ex.~5.4(iii)]{Heinz_0} (see Section~\ref{sec:examples} for further examples and discussion). This situation therefore calls for further research aimed at answering the question:
{\em Can a global convergence rate for the DR algorithm and its variants be established or estimated for some reasonable class of convex sets without the above mentioned regularity condition?}

 { The goal of this paper is to provide some partial answers to the above question, as well as  giving simple tools for establishing sublinear or linear convergence of the Douglas--Rachford algorithm and  variants. Our analysis is performed within the much more general setting of \emph{fixed point iterations}
  described by \emph{averaged nonexpansive operators}. This broad framework covers many iterative fixed-point methods including various Krasnoselskii--Mann iterations, the cyclic projection algorithm,  Douglas--Rachford algorithms and forward-backward splitting methods,
   and can be used to solve not only convex feasibility problems but also convex optimization problems and variational inequality problems. We pay special attention to the case in which the underlying sets are \emph{convex semi-algebraic sets} in a finite dimensional space. Such sets comprise a broad sub-class of convex sets that we shall show satisfy \emph{H\"older regularity properties} without requiring any further assumptions. Indeed, they capture all polyhedra and convex sets described by convex quadratic functions. Furthermore, convex semi-algebraic structure can often be relatively easily identified.}

\subsection{Content and structure of the paper}
 The detailed contributions of this paper are summarized as follows:
\begin{enumerate}[(I)]
 \item We study an abstract algorithm which we refer to as the \emph{quasi-cyclic algorithm}. This algorithm covers many iterative fixed-point methods including various Krasnoselskii--Mann iterations, the cyclic projection algorithm,  Douglas--Rachford algorithms and forward-backward splitting methods.
 In the presence of so-called \emph{bounded H\"older regularity properties}, sublinear  convergence of the algorithm is then established (Theorem~\ref{thm:abstract convergence rate}).

 \item The quasi-cyclic algorithm framework is then specialized to the Douglas--Rachford algorithm and its variants (Section~\ref{sec:DR}). We show the results apply, for instance, to the important case of feasibility problems for which the underlying sets are convex semi-algebraic in a finite dimensional space.

 \item A damped variant of the Douglas--Rachford algorithm is examined. Again, in the case in which the underlying sets are convex basic semi-algebraic sets in a finite dimensional space, we obtain a more {\em explicit estimate of the sublinear convergence rate} in terms of the dimension of the underlying space and the maximum degree of the polynomials  involved (Theorem~\ref{TheMTR:3}).
\end{enumerate}

The remainder of the paper is organized as follows: in Section~\ref{sec:preliminaries} we recall definitions and key facts used in our analysis. In Section~\ref{sec:averaged} we investigate the rate of convergence of the \emph{quasi-cyclic algorithm}. In Section~\ref{sec:DR} we specialize these results to the classical Douglas--Rachford algorithm and its cyclic variants. In Section~\ref{sec:damped DR} we consider a damped version of the Douglas--Rachford algorithm. In Section~\ref{sec:examples} we establish  explicit convergence rates for two illustrative problems. We conclude the paper in Section~\ref{sec:conclusion} by discussing possible directions for future research.

\section{Preliminaries}\label{sec:preliminaries}
  Throughout this paper our setting is a  (real) \emph{Hilbert space} $H$ with inner product $\langle \cdot, \cdot\rangle$. The \emph{induced norm} is defined by $\|x\|:=\sqrt{\langle x, x\rangle}$ for all $x \in H$. Given a closed convex subset $A$ of $H$, the \emph{(nearest point) projection} operator is the operator ${\rm P_A}:H\to A$ given by
   $${\rm P_A}x=\argmin_{a\in A}\|x-a\|.$$
 Let us now recall various definitions and facts used throughout this work, beginning with the notion of \emph{Fej\'er monotonicity}.
\begin{definition}[Fej\'{e}r monotonicity]
Let $A$ be a non-empty convex subset of a Hilbert space $H$. A sequence $(x_k)_{k\in\NN}$ in
$H$ is \emph{Fej\'{e}r
monotone} with respect to $A$ if, for all $a\in A$, we have
\begin{align*}
\|x_{k+1}-a\|\leq\|x_k-a\|\quad\forall k\in\NN.
\end{align*}
\end{definition}

\begin{fact}[{Shadows of Fej\'{e}r monotone sequences \cite[Th.~5.7(iv)]{BB}}] \label{FactPr:3}
Let $A$ be a non-empty closed  convex subset of a Hilbert space $H$ and let $(x_k)_{k\in\NN}$ be Fej\'{e}r monotone with respect to $A$. Then ${\rm P_A}(x_k)\to x$, in norm, for some $x\in A$.
\end{fact}

\begin{fact}[Fej\'er monotone convergence {\cite[Th.~3.3(iv)]{Heinz_Set}}] \label{FactPr:2}
Let $A$ be a non-empty closed  convex subset of a Hilbert space $H$ and let $(x_k)_{k\in\NN}$ be Fej\'{e}r
monotone with respect to $A$ with $x_k\to x\in A$, in norm. Then $\|x_k-x\|\leq2\di(x_k,A)$.
\end{fact}

 We now turn our attention to a H\"older regularity property for typically finite collections of sets.

\begin{definition}[Bounded H\"older regular intersection]\label{def:holder regular}
 Let $\{C_j\}_{j\in\mathbb{J}}$ be a collection of closed convex subsets in a Hilbert space $H$ with non-empty intersection. The collection $\{C_j\}_{j\in\mathbb{J}}$ has a \emph{bounded H\"older regular intersection} if, for each bounded set $K$, there exists an exponent $\gamma\in(0,1]$ and a scalar $\beta>0$ such that
  $$\di\left(x,\cap_{j\in\mathbb{J}} C_j\right)\leq \beta \left(\max_{j\in\mathbb{J}} d(x,C_j)\right)^\gamma\quad \forall x\in K.$$
 Furthermore, if the exponent $\gamma$ does not depend on the set $K$, we say the collection $\{C_j\}_{j\in\mathbb{J}}$ is \emph{bounded H\"older regular with uniform exponent $\gamma$}.
 \end{definition}

 It is clear, from Definition~\ref{def:holder regular}, that any collection containing only a single set trivially has a bounded H\"older regular intersection with uniform
 exponent $\gamma=1$. More generally, Definition~\ref{def:holder regular} with $\gamma=1$ is well-studied in the literature where it appears, amongst other names, as \emph{bounded linear regularity} \cite{BB}. For a recent study, the reader is referred to \cite[Remark~7]{KrugerThao}. The local counterpart to Definition~\ref{def:holder regular} has been characterized in \cite[Th.~1]{KrugerThao} under the name of \emph{metric $[\gamma]$-subregularity}.

 We next turn our attention to a nonexpansivity notion for operators.
\begin{definition}\label{def:expansivity prop}
An operator $T \colon H\to H$  is:
\begin{enumerate}[(a)]
\item \emph{non-expansive} if,
for all $x,y\in H$, $$\|T(x)-T(y)\| \le \|x-y\|;$$

\item \emph{firmly non-expansive}  if,
for all $x,y\in H$,
   \[
 \|T(x)-T(y)\|^2 + \|(I-T)(x)-(I-T)(y)\|^2 \le \|x-y\|^2;
\]
\item\label{def:averaged} \emph{$\alpha$-averaged} for some $\alpha \in (0,1)$, if there exists a non-expansive mapping $R \colon H\to H$ such that
  $$T=(1-\alpha)I+\alpha R.$$
\end{enumerate}
\end{definition}

The class of firmly non-expansive mappings comprises precisely the 1/2-averaged
mappings, and any $\alpha$-averaged operator is non-expansive \cite[Ch.~4]{BC2011}. The term “av-
eraged mapping” was coined in \cite{BaillonBruckReich}. The following fact provides a characterization of averaged maps that is useful for our purposes.

\begin{fact}[Characterization of averaged maps {\cite[Prop.~4.25(iii)]{BC2011}}]\label{fact:averaged}
 Let  $T \colon H\to H$  be an $\alpha$-averaged operator on a Hilbert space with $\alpha\in(0,1)$. Then, for all $x,y\in H$,
   \[
 \|T(x)-T(y)\|^2 + \frac{1-\alpha}{\alpha}\|(I-T)(x)-(I-T)(y)\|^2 \le \|x-y\|^2.
\]
\end{fact}
Denote the set of \emph{fixed points} of an operator  $T \colon H\to H$ by
\[
{\rm Fix\,}T=\{x \in H \mid T(x)=x\}.
\]
The following definition is of  a H\"older regularity property for operators.
\begin{definition}[Bounded H\"older regular operators]\label{def:holder regular operators}
 An operator  $T \colon H\to H$  is \emph{bounded  H\"{o}lder regular} if, for each bounded set $K\subseteq H$,
there exists an exponent $\gamma \in (0,1]$ and a scalar $\mu>0$ such that
\[
d(x,{\rm Fix\, }T) \le \mu \|x-T(x)\|^{\gamma}\quad \forall x\in K.
\]
Furthermore, if the exponent $\gamma$ does not depend on the set $K$, we say that $T$ is \emph{
bounded H\"older regular with uniform exponent $\gamma$}.
\end{definition}

 Note that, in the case when $\gamma=1$, Definition~\ref{def:holder regular operators} collapses to the well studied concept of bounded linear regularity \cite{BB} and has been  used in \cite{BNP} to analyze linear convergence of algorithms involving non-expansive mappings.
Moreover, it is also worth noting that  if an operator $T$ is bounded  H\"{o}lder regular with exponent $\gamma \in (0,1]$ then the mapping
$x \mapsto x-T(x)$ is bounded H\"{o}lder metric subregular with exponent $\gamma$.  H\"{o}lder metric subregularity -- which is
a natural extension of  metric subregularity -- along with H\"older type error bounds has  recently been studied in \cite{Li_Boris,Li0,Li1,Kruger}.

 Finally, we recall the definitions of  \emph{semi-algebraic functions} and \emph{semi-algebraic sets}.
 \begin{definition}[Semi-algebraic sets and functions {\cite{real}}]
 A set $D \subseteq \mathbb{R}^n$ is  \emph{semi-algebraic} if
\begin{equation}\label{eq:semi alg set}
 D: = \bigcap_{j=1}^{s}\,\bigcup_{i=1}^{l} \{x \in \RR^n\mid f_{ij}(x)=0, h_{ij}(x) < 0\}
\end{equation}
for integers $l,s$ and polynomial functions $f_{ij}, \,
h_{ij}$ on $\mathbb{R}^{n}$ $(1 \le i\le l, \, 1 \le j \le s)$.
A mapping $F \colon \mathbb{R}^{n} \rightarrow \mathbb{R}^p$ is
said to be \emph{semi-algebraic} if its \emph{graph}, ${\rm gph}F:=\{(x,F(x))\mid x \in
\mathbb{R}^n\}$, is a semi-algebraic set in $\mathbb{R}^n \times \mathbb{R}^p$.
\end{definition}

 The next fact summarises some  fundamental properties of semi-algebraic sets and functions.

\begin{fact}[Properties of semi-algebraic sets/functions]\label{fact:6}
  The following statements hold.
\begin{enumerate}[(P1)]
\item Any polynomial is a semi-algebraic function.

\item Let $D$ be a semi-algebraic
set. Then $\di(\cdot,D)$ is a
semi-algebraic function.
\item If $f,\, g$ are
semi-algebraic functions on $\mathbb{R}^n$ and $\lambda \in \mathbb{R}$
then $f+g$, $\lambda f$, $\max\{f,g\}$, $fg$ are  semi-algebraic.
\item  If $f_i$ are semi-algebraic functions, $i=1,\ldots,m$,
 and $\lambda \in \mathbb{R}$, then the sets $\{x\mid f_i(x)=\lambda, i=1,\ldots,m\}$,
 $\{x\mid f_i(x) \le \lambda, \, i=1,\ldots,m\}$ are
semi-algebraic sets.
\item If $F:\mathbb{R}^n \rightarrow \mathbb{R}^p$ and $G:\mathbb{R}^p \rightarrow \mathbb{R}^q$ are semi-algebraic mappings, then
their composition $G \circ F$ is also a semi-algebraic mapping.
\item ({\L}ojasiewicz's inequality) If
$\phi,\psi$ are two continuous semi-algebraic functions on
a compact semi-algebraic set $K \subseteq \mathbb{R}^n$ such that
$\emptyset \neq \phi^{-1}(0) \subseteq \psi^{-1}(0)$ then there
exist constants $c>0$ and $\tau \in (0,1]$ such that
\[
 |\psi(x)| \le  c |\phi(x)|^{\tau}\quad \forall x \in K.
\]
\end{enumerate}
\end{fact}
\begin{proof}
 (P1) and (P4) follow directly from the definitions. See \cite[Prop.~2.2.8]{real} for (P2), \cite[Prop.~2.2.6]{real} for (P3) and (P5), and \cite[Cor.~2.6.7]{real} for (P6).
\end{proof}

\begin{definition}[Basic semi-algebraic convex sets in $\mathbb{R}^n$]
 A set $C \subseteq \mathbb{R}^n$ is  a \emph{basic semi-algebraic convex} set if there exist $\gamma \in\NN$ and
convex polynomial functions, $g_{j}, j=1,\ldots,\gamma$ such that $C=\{x  \in \mathbb{R}^n\mid g_{j}(x) \le 0, j=1,\cdots,\gamma\}.$
\end{definition}
 Any basic semi-algebraic convex set is clearly  convex and semi-algebraic. On the other hand, there exist sets which are both convex and semi-algebraic but fail to be basic semi-algebraic convex set, see \cite{Bor_Li_Yao}.

 It transpires out that any finite collection of basic semi-algebraic convex sets has an intersection which is  boundedly H\"older regular with uniform exponent (without requiring further regularity assumptions). In the following lemma, $B(n)$ denotes the \emph{central binomial coefficient} with respect to $n$ given by $\binom{n}{[n/2]}$ where $[\,\cdot\,]$ denotes the integer part of a real number.
\begin{lemma}[H\"older regularity of basic semi-algebraic convex sets in $\mathbb{R}^n$ {\cite{Bor_Li_Yao}}]\label{ThesumSet:1}
Let $C_i$ be  basic convex semi-algebraic sets in $\mathbb{R}^n$ given by
$C_i=\{x \in \mathbb{R}^n \mid g_{ij}(x) \le 0, j=1,\ldots,m_i\}, i=1,\ldots,m$ where $g_{ij}$ are convex polynomials on $\mathbb{R}^n$ with degree at most $d$.
Let $\theta>0$ and $K\subseteq\RR^n$ be a compact set.
  Then  there exists  $c> 0$ such that
$$\di^{\theta} (x,C) \le c \left(\sum_{i=1}^m \di^{\theta}(x, C_i)\right)^{\gamma} \quad \forall x \in K,$$
where $\gamma=\left[\min\left\{\frac{(2d-1)^n+1}{2},\, B(n-1)d^n\right\}\right]^{-1}$.
\end{lemma}

We also recall the following useful recurrence relationship established in \cite{Bor_Li_Yao}.
\begin{lemma}[Recurrence relationship \cite{Bor_Li_Yao}]\label{recur} Let $p >0$, and let $\{\delta_t\}_{t\in\NN}$ and $\{\beta_t\}_{t\in\NN\textsl{•}}$ be two sequences of nonnegative numbers such that
\[
\beta_{t+1} \le \beta_t(1-\delta_t \beta_t^{p})\quad\forall t\in\NN.
\]
Then
\begin{equation*}
\beta_{t}\le\bigg(\beta_0^{-p}+\displaystyle p \sum_{i=0}^{t-1} \delta_i \bigg)^{-\frac{1}{p}}\quad\forall t\in\mathbb N,
\end{equation*} where the convention that $\frac{1}{0}=+\infty$ is adopted.
\end{lemma}

\section{The rate of convergence of the quasi-cyclic algorithm}\label{sec:averaged}

 In this section we investigate the rate of convergence of an abstract algorithm we call \emph{quasi-cyclic}. To define the algorithm, let $J$ be a finite set, and let $\{T_j\}_{j \in J}$ be a finite family of operators on a Hilbert space $H$. Given an initial point $x^0\in H$, the quasi-cyclic algorithm generates a sequence according to
\begin{equation}\label{eq:quasi_cyclic} x^{t+1} = \sum_{j \in J}w_{j,t}T_j(x^t)\quad\forall t \in \mathbb{N},\end{equation}
 for appropriately chosen weights $w_{j,t}\in\mathbb{R}$.

 The quasi-cyclic algorithm appears in \cite{BNP} where linear convergence of the algorithm was established under suitable regularity conditions. As we shall soon see, the quasi-cyclic algorithm provides a broad framework which covers many important existing algorithms including Douglas-Rachford algorithms, the cyclic projection algorithm, the Krasnoselskii--Mann method, and forward-backward splitting. To establish its convergence rate, we use three preparatory results.
\begin{lemma}\label{lem:asym}
 Let  $J$ be a finite set and let $\{T_j\}_{j \in J}$ be a finite family of
$\alpha$-averaged operators on a Hilbert space $H$ with $\cap_{j \in J}{\rm Fix\, }T_j \neq \emptyset$ and let $\alpha\in (0,1)$. For each $t \in \mathbb{N}$, let $w_{j,t} \in \mathbb{R}$, $j \in J$, be such that $w_{j,t}\ge 0$ and $\sum_{j \in J}w_{j,t}=1$.
Let $x^0 \in H$ and consider the quasi-cyclic algorithm generated by \begin{equation}
x^{t+1} = \sum_{j \in J}w_{j,t}T_j\big(x^t\big)\quad\forall t \in \mathbb{N}.\end{equation}  Suppose that $$\sigma:=\displaystyle \inf_{t \in \mathbb{N}} \inf_{j \in J_+(t)}\{w_{j,t}\}>0 \mbox{ where } J_+(t):=\{j \in J: w_{j,t}>0\} \mbox{ for each } t \in \mathbb{N}.$$
Then $\{x^t\}_{t \in \NN}$ is Fej\'er monotone with respect to $\cap_{j \in J} {\rm Fix} T_j$,
$\{\di(x^t,\cap_{j \in J}{\rm Fix\, }T_j)\}_{t\in\NN}$ is nonincreasing (and hence convergent)
and $\max_{j\in J_+(t)}\|x^t-T_j(x^t)\|\to0$ as $t\to\infty$.
\end{lemma}
\begin{proof}
 Let $y\in\cap_{j\in J}{\rm Fix}T_j$. Then, for all $t\in\NN$, convexity of $\|\cdot\|^2$ yields
  \begin{equation}\label{eq:lemma averaged}
   \|x^{t+1}-y\|^2=\|\sum_{j\in J}w_{j,t}T_j(x^t)-y\|^2\leq \sum_{j\in J}w_{j,t}\|T_j(x^t)-y\|^2\leq \|x^t-y\|^2,
  \end{equation}
  where the last inequality follows by the fact that each $T_j$ is $\alpha$-averaged (and so, is nonexpansive).
 Thus, $\{x^t\}_{t \in NN}$ is Fej\'er monotone with respect to $\cap_{j\in J}{\rm Fix}T_j$ and  $\{\|x^t-y\|^2\}_{t\in\NN}$ is a nonnegative, decreasing sequence and hence convergent. Furthermore, from \eqref{eq:lemma averaged} we obtain
 \begin{equation}\label{eq:asym1}
  \lim_{t\to \infty}\sum_{j\in J}w_{j,t}\|T_j(x^t)-y\|^2=\lim_{t\to\infty}\|x^t-y\|^2.
 \end{equation}
Since $T_j$ is $\alpha$-averaged for each $j\in J$, Fact~\ref{fact:averaged} implies, for all $t\in\NN$,
  \begin{equation*}
   \|T_j(x^t)-y\|^2+\frac{1-\alpha}{\alpha}\|x^t-T_j(x^t)\|^2\leq \|x^t-y\|^2,
  \end{equation*}
 from which, for sufficiently large $t$ , we deduce
  \begin{align*}
  \|x^t-y\|^2-\sum_{j\in J}w_{j,t}\|T_j(x^t)-y\|^2
  &\geq \frac{1-\alpha}{\alpha}\sum_{j\in J}w_{j,t}\|x^t-T_j(x^t)\|^2 \\
   &\geq \frac{1-\alpha}{\alpha}\sigma\max_{j\in J_+(t)}\|x^t-T_j(x^t)\|^2.
  \end{align*}
 Together with \eqref{eq:asym1} this gives $\max_{j\in J_+(t)}\|x^t-T_j(x^t)\|\to 0$ as claimed.
\end{proof}

The following proposition provides a convergence rate for Fej\'er monotone sequences  satisfying an additional property, which we  later show to be satisfied in the presence of H\"older regularity.

\begin{proposition}\label{prop:convergence rate}
 Let $F$ be a non-empty closed convex set in a Hilbert space $H$ and let $s$ be a positive integer. Suppose the sequence $\{x^t\}$ is Fej\'er monotone with respect to $F$ and satisfies
  \begin{equation}\label{eq:dist recurrence}
   \di^2(x^{(t+1)s},F) \le \di^2(x^{ts},F)-\delta\, \di^{2\theta}(x^{ts},F),\; \forall \;t\in\mathbb{N},
  \end{equation}
 for some $\delta>0$ and $\theta\geq 1$. Then $x^t \rightarrow \bar x$ for some $\bar x \in F$
 {and,  there exist constants $M_1,M_2\geq0$ and $r\in[0,1)$ such that
\[
\|x^{t}-\bar x\| \le \begin{cases}
                        M_1\, t^{-\frac{1}{2(\theta-1)}} &  \theta>1, \\
                        M_2\, r^t &  \theta=1. \\
                       \end{cases}\
\]
Furthermore, the constants may be chosen to be
 \begin{equation}\label{eq:explicit constants}
  \left\{\begin{array}{ccl}
   M_1&:=& 2 \max\{(2s)^{\frac{1}{2(\theta-1)}}\,\left[(\theta-1) \delta\right]^{-\frac{1}{2(\theta-1)}},(2s)^{\frac{1}{2(\theta-1)}}\, \di(x^{0},F)\} \\ M_2&:=&2\max\{\left(\sqrt[4s]{1-\delta}\right)^{-2s} \di (x^{0},F),\sqrt{\di (x^{0},F)}\} \\
   r&:=&\sqrt[4s]{1-\delta}, \end{array} \right.
  \end{equation}
 and $\delta$ necessarily lies in $(0,1]$ whenever $\theta=1$.}
\end{proposition}

\begin{proof} Without loss of generality, we assume that $x^{0} \notin F$.
  Let $\beta_t:=\di^2 (x^{ts},F)$ and $p:=\theta-1\geq 0$. Then \eqref{eq:dist recurrence} becomes
\begin{equation}\label{eq:dist recurrence 2}
\beta_{t+1} \le \beta_{t}\bigg(1- \delta \beta_t^{p}\bigg)
\end{equation}
We now distinguish two cases based on the value of $\theta$.

 {\em Case~1:} Suppose $\theta\in (1,+\infty)$. Then, noting that $1/(\theta-1)>0$, Lemma~\ref{recur} implies
\[
\beta_{t}\le\bigg(\beta_0^{-p}+\displaystyle (\theta-1) \delta t \bigg)^{-\frac{1}{\theta-1}}\le \bigg(0+\displaystyle (\theta-1) \delta t \bigg)^{-\frac{1}{\theta-1}}  \;\mbox{ for all }\;t\in\mathbb N.
\]
It follows that
$\di (x^{ts},F)=\sqrt{\beta_t} \le \left[(\theta-1) \delta\right]^{-\frac{1}{2(\theta-1)}} t^{-\frac{1}{2(\theta-1)}}$.
In particular, we have $\|x^{ts}-{\rm P}_{F} (x^{ts})\|=\di (x^{ts},F) \rightarrow 0$. By Fact~\ref{FactPr:3}, ${\rm P}_{F} (x^{ts}) \rightarrow \bar x$ for some $\bar x \in F$
and hence $x^{ts} \rightarrow \bar x \in F$ as $t \rightarrow \infty$. Denote $$\bar M_1:=\max\{(2s)^{\frac{1}{2(\theta-1)}}\,\left[(\theta-1) \delta\right]^{-\frac{1}{2(\theta-1)}},(2s)^{\frac{1}{2(\theta-1)}}\, \di(x^{0},F)\}.$$
On one hand, if $t \le 2s$, then
\[
\di (x^{t},F) \le \di(x^{0},F) =  \big[(2s)^{\frac{1}{2(\theta-1)}}\, \di(x^{0},F)\big] (2s)^{-\frac{1}{2(\theta-1)}} \le \bar M_1 t^{-\frac{1}{2(\theta-1)}},
\]
and, on the other hand, if $t>2s$ (and so, $\frac{t}{s}-1 \ge \frac{t}{2s})$, then
\begin{eqnarray*}
\di (x^{t},F) \le \di (x^{s \lfloor \frac{t}{s} \rfloor},F) &\le&  \left[(\theta-1) \delta\right]^{-\frac{1}{2(\theta-1)}} \big(\lfloor \frac{t}{s} \rfloor\big)^{-\frac{1}{2(\theta-1)}} \\
&\le &  \left[(\theta-1) \delta\right]^{-\frac{1}{2(\theta-1)}} \big( \frac{t}{s}-1\big)^{-\frac{1}{2(\theta-1)}} \\
&\le &   \left[(\theta-1) \delta\right]^{-\frac{1}{2(\theta-1)}} \big( \frac{t}{2s}\big)^{-\frac{1}{2(\theta-1)}} \\
& \le & \bar M_1 \,  t^{-\frac{1}{2(\theta-1)}}.
\end{eqnarray*}
Here $\lfloor \frac{t}{s} \rfloor$ denotes the largest integer which is smaller or equal to $\frac{t}{s}$, the first inequality follows from
the Fej\'er monotonicity of $\{x^t\}$ and the last inequality follows from the definition of $\bar M_1$.
This, together with Fact~\ref{FactPr:2}, implies that
\[
\|x^t-\bar x\|  \le 2\di (x^{t},F) \le 2 \bar M_1 t^{-\frac{1}{2(\theta-1)}}=M_1 t^{-\frac{1}{2(\theta-1)}},
\]
where the last equality follows from the definition of $M_1$.

 {\em Case~2:} Suppose $\theta=1$. Then \eqref{eq:dist recurrence 2} simplifies to
     $\beta_{t+1} \le (1-\delta) \beta_t\text{ for all }t\in\NN.$
Moreover, this shows that $\delta \in (0,1]$ and
 that
\[
\di (x^{ts},F)=\sqrt{\beta_t} \le \sqrt{\beta_0} \left(\sqrt{1-\delta}\right)^{t} .
\]
Let $\bar M_2=\max\{\left(\sqrt[4s]{1-\delta}\right)^{-2s} \di (x^{0},F),\sqrt{\di (x^{0},F)}\}$. 
On one hand, if $t \le 2s$, then
\[
\di (x^{t},F) \le \di(x^{0},F) =\big[\left(\sqrt[4s]{1-\delta}\right)^{-2s} \di (x^{0},F)\big] \left(\sqrt[4s]{1-\delta}\right)^{2s} \le \bar M_2  \left(\sqrt[4s]{1-\delta}\right)^{t},
\]
and, on the other hand, if $t>2s$ (and so, $\frac{t}{s}-1 \ge \frac{t}{2s})$, then
\begin{eqnarray*}
\di (x^{t},F) \le \di (x^{s \lfloor \frac{t}{s} \rfloor},F) &\le&  \sqrt{\beta_0} \left(\sqrt{1-\delta}\right)^{\lfloor \frac{t}{s} \rfloor} \\
&\le &  \sqrt{\beta_0} \left(\sqrt{1-\delta}\right)^{\frac{t}{s}-1} \\
&\le &   \sqrt{\beta_0} \left(\sqrt{1-\delta}\right)^{\frac{t}{2s}} \\
& \le & \sqrt{\beta_0}  \left(\sqrt[4s]{1-\delta}\right)^{t}  \le \bar M_2 \left(\sqrt[4s]{1-\delta}\right)^{t}.
\end{eqnarray*}
%
%
By the same argument as used in Case~1, for some $\bar x\in F$, we see that
\[
\|x^t-\bar x\|  \le 2\di (x^{t},F) \le 2\bar M_2 \left(\sqrt[4s]{1-\delta}\right)^{t}=M_2 \left(\sqrt[4s]{1-\delta}\right)^{t}.
\]
The conclusion follows by setting $r=\sqrt[4s]{1-\delta} \in [0,1).$
\end{proof}

We are now in a position to state our main convergence result, which we simultaneously prove for both variants of our H\"older regularity assumption (non-uniform and uniform versions).

\begin{theorem}[Rate of convergence of the quasi-cyclic algorithm]\label{thm:abstract convergence rate}
Let  $J$ be a finite set and let $\{T_j\}_{j \in J}$ be a finite family of
$\alpha$-averaged operators on a Hilbert space $H$ with $\cap_{j \in J}{\rm Fix\, }T_j \neq \emptyset$ and $\alpha\in (0,1)$. For each $t \in \mathbb{N}$, let $w_{j,t} \in \mathbb{R}$, $j \in J$, be such that $w_{j,t}\ge 0$ and $\sum_{j \in J}w_{j,t}=1$.
Let $x^0 \in H$ and consider the
 quasi-cyclic algorithm generated by~\eqref{eq:quasi_cyclic}.
Suppose the following assumptions hold:
 \begin{enumerate}[(a)]
  \setlength{\itemsep}{0pt}
  \setlength{\parskip}{0pt}
  \item For each $j\in J$, the operator $T_j$ is bounded H\"older regular .
  \item $\{{\rm Fix\ }T_j\}_{j \in J}$ has a boundedly H\"{o}lder regular intersection.
  \item $\sigma:=\displaystyle \inf_{t \in \mathbb{N}} \inf_{j \in J_+(t)}\{w_{j,t}\}>0$ where $J_+(t):=\{j \in J: w_{j,t}>0\}$ for each $t \in \mathbb{N}$, and there exists an $s\in\mathbb{N}$ such that
   $$ J_+(t)\cup J_+(t+1)\cup\dots\cup J_{+}(t+s-1)=J,\quad\forall t\in\mathbb{N}. $$
\end{enumerate}
Then $x^t \to \bar x\in  \cap_{j \in J}{\rm Fix\, }T_j \neq \emptyset$ at least with a sublinear rate $O(t^{-\rho})$ for some $\rho>0$.

In particular, if we assume  ${\rm (a')}$, ${\rm (b')}$ and ${\rm (c)}$ hold where ${\rm (a')}$, ${\rm (b')}$ are given by
\begin{enumerate}[(a$'$)]
  \setlength{\itemsep}{0pt}
  \setlength{\parskip}{0pt}
   \item for each $j\in J$, the operator $T_j$ is bounded H\"older regular with uniform exponent $\gamma_{1,j}\in(0,1]$;
   \item $\{{\rm Fix\ }T_j\}_{j \in J}$ has a bounded H\"{o}lder regular intersection with  uniform exponent $\gamma_2\in(0,1]$,
\end{enumerate}{
then there exist constants $M_1,M_2\geq 0$ and $r\in[0,1)$ such that
\[
\|x^{t}-\bar x\| \le \begin{cases}
                       M_1 t^{-\frac{\gamma}{2(1-\gamma)}}, & \gamma \in (0,1), \\
                       M_2 \, r^t,  & \gamma=1,
                     \end{cases}
\]
where $\gamma:=\gamma_1 \gamma_2$ and $\gamma_1:=\min\{\gamma_{1,j}\mid j\in J\}$.}
\end{theorem}

\begin{proof}
 Denote $F:= \cap_{j \in J}{\rm Fix\, }T_j$. We first consider the case in which Assumptions~(a), (b) and (c) hold. We first observe that, as a consequence of Lemma~\ref{lem:asym}, we may assume without loss of generality the following two inequalities holds:
  \begin{align}
   \max_{j\in J_+(t)}\|x^t-T_j(x^t)\| &\leq 1,\quad\forall t\in\NN, \label{eq:less than 1} \\
    \di^2(x^{ts},F)- \di^2(x^{(t+1)s},F) &\leq 1,\quad\forall t\in\NN.    \label{eq:less than 1 b}
  \end{align}

 Now, let $K$ be a bounded set such that $\{x^t \mid t\in\NN\}\subseteq K$. For each $j\in J$, since the operator $T_j$ is bounded H\"{o}lder regular, there exist exponents $\gamma_{1,j}>0$ and scalars $\mu_j>0$ such that
\begin{equation}\label{eq:use00}
\di(x,{\rm Fix\,} T_j) \le \mu_j \|x-T_j(x)\|^{\gamma_{1,j}}\quad\forall x \in K.
\end{equation}
Set $\gamma_1=\min\{\gamma_{1,j}\mid j\in J\}$ and $\mu:=\max\{\mu_j \mid j \in J\}$. By \eqref{eq:use00} and \eqref{eq:less than 1}, for all $j\in J_+(t)$, it follows that
 \begin{equation}\label{eq:use01}
\di(x^t,{\rm Fix\,} T_j) \le \mu_j \|x^t-T_j(x^t)\|^{\gamma_{1,j}} \le \mu \|x^t-T_j(x^t)\|^{\gamma_{1}}\quad\forall x \in K.
\end{equation}
Also, since $\{{\rm Fix\ }T_j\}_{j \in J}$ has a boundedly H\"{o}lder regular intersection, there exist an exponent $\gamma_2>0$ and a scalar $\beta>0$ such that
\begin{equation}\label{eq:use02}
  \di(x,F) \leq \beta \left(\max_{j\in J}\di(x,{\rm Fix\,}T_j)\right)^{\gamma_2}\quad\forall x\in K.
\end{equation}

 Fix an arbitrary index $j'\in J$. Assumption~(c) ensures that, for any $t\in\mathbb{N}$, there exists index $t_k\in\{ts,\dots,(t+1)s-1\}$ such that $j'\in J_+(t_k)$. Then
 \begin{equation}\label{eq:upp 00}\begin{split}
  \di^2(x^{ts},{\rm Fix\,} T_{j'})
   &\leq \left(\di(x^{t_k},{\rm Fix\,} T_{j'})+\|x^{ts}-x^{t_k}\|\right)^2 \\
   &\leq\left(\di(x^{t_k},{\rm Fix\,} T_{j'})+\sum_{n=ts}^{t_k-1}\|x^n-x^{n+1}\|\right)^2 \\
   &\leq (t_k-ts+1)\left(\di^2(x^{t_k},{\rm Fix\,} T_{j'})+\sum_{n=ts}^{t_k-1}\|x^n-x^{n+1}\|^2\right) \\
   &\leq s\left(\mu\left(\|x^{t_k}-T_{j'}(x^{t_k})\|^{2}\right)^{\gamma_1}+\sum_{n=ts}^{(t+1)s-1}\|x^n-x^{n+1}\|^2\right),
 \end{split}\end{equation}
 where the second from last inequality follows from convexity of the function $(\cdot)^2$, and the last uses \eqref{eq:use01} noting that $j'\in J_+(t_k)$.

  Since each $T_j$ is $\alpha$-averaged, for all $t\in\NN$, the convex combination $\sum_{j\in J}w_{t,j}T_j$ is $\alpha$-averaged (and, in particular, nonexpansive), and hence, for all $x\in H$ and $y \in F$, we have
\begin{eqnarray}\label{eq:089}
\|\sum_{j \in J}w_{j,t}T_j(x)-y\|^2 &=&  \|\sum_{j \in J}w_{j,t}\big(T_j(x)-y\big)\|^2 \nonumber \\
& \le  & \sum_{j \in J}w_{j,t}\|T_j(x)-y\|^2 \nonumber \\
&\le & \sum_{j \in J}w_{j,t}\big(\|x-y\|^2-\frac{1-\alpha}{\alpha}\|x-T_j(x)\|^2\big) \nonumber \\
& = & \|x-y\|^2- \frac{1-\alpha}{\alpha}\sum_{j \in J}w_{j,t}\|x-T_j(x)\|^2 \\
& \le & \|x-y\|^2-  \sigma\left(\frac{1-\alpha}{\alpha}\right)\|x-T_{j}(x)\|^2 \quad\forall j\in J_+(t). \nonumber
\end{eqnarray}
We therefore have that
   \begin{equation}\label{eq:upp 01}\begin{split}
   \sigma\,\frac{1-\alpha}{\alpha}\,\|x^{t_k}-T_{j'}(x^{t_k})\|^{2}
   &\leq \|x^{t_k}-P_F(x^{ts})\|^2-\|x^{t_k+1}-P_F(x^{ts})\|^2 \\
   &\leq \|x^{ts}-P_F(x^{ts})\|^2-\|x^{(t+1)s}-P_F(x^{ts})\|^2 \\
   &\leq \di^2(x^{ts},F)- \di^2(x^{(t+1)s},F).
  \end{split}\end{equation}
Furthermore, for each $n\in\{ts,\ldots,(t+1)s-1\}$, applying $x=x^n$ and $y=P_F(x^{ts})$ in \eqref{eq:089}
we have
\begin{eqnarray*}
\frac{1-\alpha}{\alpha} \|x^n-x^{n+1}\|^2 &=& \frac{1-\alpha}{\alpha} \|x^n-\sum_{j \in J}w_{j,n}T_j(x^n)\|^2 \\ &\le & \frac{1-\alpha}{\alpha} \sum_{j \in J}w_{j,n} \|x^n-T_j(x^n)\|^2 \\
&\le &\|x^n-P_F(x^{ts})\|^2-\|x^{n+1}-P_F(x^{ts})\|^2,
\end{eqnarray*}
and thus it follows that
  \begin{equation}\label{eq:upp 02}\begin{split}
    \frac{1-\alpha}{\alpha}\sum_{n=ts}^{(t+1)s-1}\|x^n-x^{n+1}\|^2 
      &\leq \sum_{n=ts}^{(t+1)s-1}\left(\|x^n-P_F(x^{ts})\|^2-\|x^{n+1}-P_F(x^{ts})\|^2\right)  \\
      &= \|x^{ts}-P_F(x^{ts})\|^2-\|x^{(t+1)s}-P_F(x^{ts})\|^2 \\
      &\leq \di^2(x^{ts},F)- \di^2(x^{(t+1)s},F) \\
      &\leq \left( \di^2(x^{ts},F)- \di^2(x^{(t+1)s},F) \right)^{\gamma_1},
  \end{split}\end{equation}
  where the last inequality follows from \eqref{eq:less than 1 b}. 
 Altogether, combining \eqref{eq:upp 00}, \eqref{eq:upp 01} and \eqref{eq:upp 02} gives
  \begin{equation*}\begin{split}
  \di^2(x^{ts},{\rm Fix\,} T_{j'})
   &\leq s\left(\mu\left(\frac{\alpha}{\sigma(1-\alpha)}\right)^{\gamma_1}+\frac{\alpha}{1-\alpha}\right)\left(\di^2(x^{ts},F)- \di^2(x^{(t+1)s},F)\right)^{\gamma_1}.
  \end{split}\end{equation*}
 Since $j'\in J$ was chosen arbitrary, using \eqref{eq:use02} we obtain
  \begin{equation}\label{eq:use04}\begin{split}
   \di^{2}(x^{ts},F)
   &\leq \beta^2 \max_{j\in J}\di^{2\gamma_2}(x^{ts},{\rm Fix\,}T_j)  \\
   &\leq \delta^{-1}\left(\di^2(x^{ts},F)- \di^2(x^{(t+1)s},F)\right)^{1/\theta},
 \end{split}\end{equation}
 where the constant $\delta>0$ and $\gamma>0$ are given by
  $$\delta:= \left(s^{\gamma_2}\beta^2\left(\mu\left(\frac{\alpha}{\sigma(1-\alpha)}\right)^{\gamma_1}+\frac{\alpha}{1-\alpha}\right)^{\gamma_2}\right)^{-1}\qquad \theta:=\frac{1}{\gamma_1\gamma_2}.$$
Rearranging \eqref{eq:use04} gives
 $$\di^2(x^{(t+1)s},F)\leq \di^2(x^{{ts}},F) - \delta\,\di^{2\theta}(x^{{ts}},F),$$
Then, the first assertion follows from Proposition~\ref{prop:convergence rate}.

 To establish the second assertion, we suppose that the assumptions (a$'$), (b$'$) and (c) hold. Proceed with the same proof as above, and noting that the exponents $\gamma_{1j}$ and $\gamma_2$
are now independent of the choice of $K$, we see that  the second assertion also follows.
\end{proof}

{
\begin{remark}
 A closer look at the proof of Theorem~\ref{thm:abstract convergence rate} reveals that a quantification of the constants $M_1,M_2$ and $r$ is possible using the various regularity constants/exponents and \eqref{eq:explicit constants}. More precisely, \eqref{eq:explicit constants} holds with
  $$\delta:= \left(s^{\gamma_2}\beta^2\left(\mu\left(\frac{\alpha}{\sigma(1-\alpha)}\right)^{\gamma_1}+\frac{\alpha}{1-\alpha}\right)^{\gamma_2}\right)^{-1}\qquad \theta:=\frac{1}{\gamma_1\gamma_2},\quad F:=\bigcap_{j \in J}{\rm Fix\, }T_j.$$
 Here $\mu$ is the max of the constants of bounded H\"older regularity of the individual operators $T_j$ and $\beta$ is the constant of bounded H\"older regularity of the collection $\{{\rm Fix\ }T_j\}_{j\in J}$, respectively, on an appropriate compact set. Consequently, these expressions, appropriately specialized, also hold for all the subsequent corollaries of Theorem~\ref{thm:abstract convergence rate}.
\end{remark}
}

\begin{remark}
 Theorem~\ref{thm:abstract convergence rate} generalizes \cite[Th.~6.1]{BNP} which considers the special case in which the H\"older exponents are independent of the bounded set $K$ and are specified by $\gamma_{1j}=\gamma_2=1$, $j=1,\ldots,m$.
\end{remark}

 A slight refinement of Theorem~\ref{thm:abstract convergence rate} which allows for extrapolations as well as different averaging constants is possible. More precisely, an \emph{extrapolation}  of the operator $T$ (in the sense of \cite{BCK2006}) is a (non-convex) combination of the form $wT+(1-w)I$ where the weight $w$ may take values larger than~$1$.
Recall that for a finite set $\Omega$, we use $|\Omega|$ to denote the number of elements of $\Omega$.

\begin{corollary}[Extrapolated quasi-cyclic algorithm]\label{cor:refined qca}
Let $J:=\{1,2,\dots,m\}$, let $\{T_j\}_{j \in J}$ be a finite family of
$\alpha_j$-averaged operators on a Hilbert space $H$ with $\cap_{j \in J}{\rm Fix\, }T_j \neq \emptyset$ and $\alpha_j\in (0,1)$, and let $\alpha > 0$
be such that $\max_{j\in J}\alpha_j \leq \alpha$.
 For each $t \in \mathbb{N}$, let $w_{0,t}\in\mathbb{R}$ and $w_{j,t}
\in [0,\frac{\alpha}{\alpha_j}]$ $j \in J$, be such that
 $$\sum_{j \in J}w_{j,t}+w_{0,t}=1\text{~and~}{ \sup_{t \in \NN}\{\sum_{j\in J}\alpha_jw_{j,t}\}<\alpha}.$$
Let $x^0 \in H$ and consider the extrapolated quasi-cyclic algorithm generated by
\begin{equation}\label{eq:refined qca}
 x^{t+1} = \sum_{j \in J}w_{j,t}T_j\big(x^t\big) + w_{0,t}x^t\quad\forall t \in \mathbb{N}.
\end{equation} Assume the following hypotheses.
 \begin{enumerate}[(a)]
  \setlength{\itemsep}{0pt}
  \setlength{\parskip}{0pt}
  \item For each $j\in J$, the operator $T_j$ is bounded H\"older regular.
  \item $\{{\rm Fix\ }T_j\}_{j \in J}$ has a boundedly H\"{o}lder regular intersection.
  \item $\sigma:=\displaystyle \inf_{t \in \mathbb{N}} \inf_{j \in J_+(t)}\{w_{j,t}\}>0$  where $J_+(t)=\{j \in J: w_{j,t}>0\}$ for each $t \in \mathbb{N}$, and there exists an $s\in\mathbb{N}$ such that
   $$ J_+(t)\cup J_+(t+1)\cup\dots\cup J_{+}(t+s-1)=J,\quad\forall t\in\mathbb{N}. $$
\end{enumerate}
Then $x^t \to \bar x\in  \cap_{j \in J}{\rm Fix\, }T_j \neq \emptyset$ at least with a sublinear rate $O(t^{-\rho})$ for some $\rho>0$.

In particular, if we assume  ${\rm (a')}$, ${\rm (b')}$ and ${\rm (c)}$ hold where ${\rm (a')}$, ${\rm (b')}$ are given by
\begin{enumerate}[(a$'$)]
  \setlength{\itemsep}{0pt}
  \setlength{\parskip}{0pt}
   \item for each $j\in J$, the operator $T_j$ is bounded H\"older regular with uniform exponent $\gamma_{1,j}\in(0,1]$;
   \item $\{{\rm Fix\ }T_j\}_{j \in J}$ has a bounded H\"{o}lder regular intersection with  uniform exponent $\gamma_2\in(0,1]$,
\end{enumerate}
then there exist constants $M_1,M_2\geq 0$ and $r\in[0,1)$ such that
\[
\|x^{t}-\bar x\| \le \begin{cases}
                       M_1 t^{-\frac{\gamma}{2(1-\gamma)}}, & \gamma \in (0,1), \\
                       M_2 \, r^t,  & \gamma=1,
                     \end{cases}
\]
where $\gamma:=\gamma_1 \gamma_2$ and $\gamma_1:=\min\{\gamma_{1,j}\mid j\in J\}$.
\end{corollary}
\begin{proof}
  For each $j\in J$, by Definition~\ref{def:expansivity prop}\eqref{def:averaged}, the operator $\overline{T}_j$ is $\alpha$-averaged where
 $$ \overline{T}_j := \frac{\alpha}{\alpha_j} \, T_j -\left(\frac{\alpha}{\alpha_j}-1\right)I.$$ Let $\overline w_{j,t}:=\frac{w_{j,t}\alpha_j}{\alpha}$, $j \in J$.
 Then, $\overline w_{j,t}\ge 0$ and $$\sum_{j \in J}\overline w_{j,t}= \sum_{j \in J}\frac{w_{j,t}\alpha_j}{\alpha} {= \frac{1}{\alpha} \sum_{j \in J}w_{j,t}\alpha_j } < 1.$$
Let $\overline w_{0,t}:=1-\sum_{j \in J}\overline w_{j,t}$, {$\overline{T}_{0}(x):=x$} for all $x \in H$, and $\overline J= J \cup \{0\}$.
Then, for all $t \in \mathbb{N}$, $\sum_{j \in \overline{J}}\overline w_{j,t} =1$ with $w_{j,t} \ge 0$, $j \in \overline{J}$ and
\begin{eqnarray*}
 x^{t+1} &=&  \sum_{j \in J}w_{j,t}T_j\big(x^t\big) + w_{0,t}x^t \\
 & = & \sum_{j\in  J}\overline w_{j,t}\frac{\alpha}{\alpha_j}T_j(x^t) - \left(\sum_{j\in J}\overline w_{j,t}\frac{\alpha}{\alpha_j}-1\right)x^t \\
 & = & \sum_{j\in  J}\overline w_{j,t}\left[\frac{\alpha}{\alpha_j}T_j(x^t) - \left(\frac{\alpha}{\alpha_j} -1\right)x^t\right] +\left( 1-\sum_{j\in J}\overline{w}_{j,t}\right)x^t \\
 & = & \sum_{j\in \overline J}\overline w_{j,t}\overline T_j(x^t).
\end{eqnarray*}
Since ${\rm Fix\ } \overline{T}_j = {\rm Fix\ }  T_j$ for all $j\in J$, $\{{\rm Fix\ } \overline{T}_j\}_{j\in J}$ is bounded H\"older regular (with uniform exponent $\gamma_{2}$) if and only if $\{{\rm Fix\ } T_j\}_{j\in J}$ is bounded H\"older regular (with uniform exponent $\gamma_{2}$).
This together with ${\rm Fix\ }\overline{T}_{0}=H$ implies that $\{{\rm Fix\ }\overline{T}_j\}_{j \in \overline{J}}$ also has a boundedly H\"{o}lder regular intersection.
For all $j \in J$, $\|x-\overline{T}_jx\| = \frac{\alpha}{\alpha_j}\|T_jx-x\|$ and so, $\overline T_j$ is bounded H\"older regular (with uniform exponent $\gamma_{1,k}$) if and only if $T_j$ is bounded H\"older regular (with uniform exponent $\gamma_{1,k}$).
Clearly, $\overline T_{l+1}$ is bounded H\"older regular with a uniform exponent $1$. Thus,
Assumption~(a) and (b) of Theorem~\ref{thm:abstract convergence rate} hold for $\{\overline w_{j,t}\}_{j \in \overline{J}}$.
Moreover, noting that by assumption $\sup_{t \in \NN}\{\sum_{j \in J} w_{j,t}\alpha_j\}<\alpha$, we have
\[
\inf_{t\in \NN} \overline w_{0,t}= \inf_{t\in \NN}\{1-\sum_{j \in J}\overline w_{j,t}\}= \inf_{t\in \NN}\{1-  \sum_{j \in J} \frac{w_{j,t}\alpha_j}{\alpha}\}=\frac{\alpha-\sup_{t \in \NN}\{\sum_{j \in J} w_{j,t}\alpha_j\} }{\alpha}>0.
\]
This together with Assumption~(c) of this corollary implies that Assumption~(c) of Theorem~\ref{thm:abstract convergence rate} holds for $\{\overline w_{j,t}\}_{j \in \overline{J}}$.
 Therefore,the claimed result now follows from Theorem~\ref{thm:abstract convergence rate}.
\end{proof}

\begin{remark}[Common fixed points]\label{re:common fixed points}
 Throughout this work we assume the collection of operators $\{T_j\}_{j\in J}$ ($J$ a finite index set) to have a common fixed point. In this setting, with appropriate nonexpansivity properties, one has that the fixed point set of convex combinations or compositions of the operators $\{T_j\}_{j\in J}$ is precisely the set of their common fixed points. Whilst there do exist several instance where such an assumption does not hold ({\em e.g.,} regularization schemes such as \cite{Luke3}), this does not preclude their analysis using the theory presented here (see Proposition~\ref{prop:RAAR}). Indeed, for such cases, the fixed point set of an appropriate convex combination or composition of operators is non-empty, and this aggregated operator thus amenable to our results (rather than the individual operators themselves). The question of usefully characterizing the fixed point set of this aggregated operator must then be addressed; a matter significantly more subtle in the absence of a common fixed point.
\end{remark}

 We next provide four important specializations of Theorem~\ref{thm:abstract convergence rate}. The first result is concerned with a simple fixed point iteration, the second with a Kransnoselskii--Mann scheme, the third with the method of cyclic projections, and the fourth with forward-backward splitting for variational inequalities.

\begin{corollary}[Averaged fixed point iterations with H\"older regularity] \label{cor:1}
 Let $T$ be an $\alpha$-averaged operators on a Hilbert space $H$ with ${\rm Fix}\,T\neq\emptyset$ and $\alpha\in (0,1)$. Suppose $T$ is bounded H\"older regular. Let $x^0 \in H$ and set $x^{t+1}=Tx^t$. Then  $x^t \to \bar x \in {\rm Fix}\,T \neq \emptyset$
  at least with a sublinear rate $O(t^{-\rho})$ for some $\rho>0$. In particular, if $T$ is bounded H\"older regular with uniform exponent $\gamma\in(0,1]$
  then exist $M>0$ and $r\in[0,1)$ such that
\[
\|x^{t}-\bar x\| \le \begin{cases}
                       M t^{-\frac{\gamma}{2(1-\gamma)}}, & \gamma \in (0,1), \\
                       M \, r^t,  & \gamma=1.
                     \end{cases}
\]
\end{corollary}
\begin{proof}
The conclusion follows immediately from Theorem~\ref{thm:abstract convergence rate}.
\end{proof}

\begin{corollary}[Krasnoselskii--Mann iterations with H\"older regularity]\label{cor:KM iteration}
 Let $T$ be an $\alpha$-averaged operator on a Hilbert space $H$ with ${\rm Fix}\,T\neq\emptyset$ and $\alpha\in(0,1)$. Suppose $T$ is bounded H\"older regular. Let $\sigma_0\in(0,1)$ and let $(\lambda_t)_{t\in\NN}$ be a sequence of real numbers with $\displaystyle \sigma_0:=\inf_{t \in \NN} \{\lambda_t (1-\lambda_t)\}>0$. Given an initial point $x^0 \in H$, set
  $$x^{t+1}=x^t+\lambda_t(Tx^t-x^t).$$
Then $x^t \to \bar x \in {\rm Fix}\,T \neq \emptyset$ at least with a sublinear rate $O(t^{-\rho})$ for some $\rho>0$.
In particular, if $T$ is bounded H\"older regular with uniform exponent $\gamma\in(0,1]$
 then there exist $M>0$ and $r\in[0,1)$ such that
\[
\|x^{t}-\bar x\| \le \begin{cases}
                       M t^{-\frac{\gamma}{2(1-\gamma)}}, & \gamma \in (0,1), \\
                       M \, r^t,  & \gamma=1.
                     \end{cases}
\]
\end{corollary}
\begin{proof}
 First observe that the sequence $(x^t)_{t\in\NN}$ is given by $x^{t+1}=T_tx^t$ where
  $T_t=(1-\lambda_t)I+\lambda_tT.$
Here, $1-\lambda_t \ge \sigma_0>0$  and $\lambda_t\geq \sigma_0>0$ for all $t \in \mathbb{N}$ by our assumption.

 A straightforward manipulation shows that the identity map, $I$, is bounded H\"older regular with uniform exponent $\gamma_{1,1}\leq 1$. Since ${\rm Fix}\,I=H$, the collection $\{{\rm Fix}\,I,{\rm Fix}\,T\}$ has a bounded H\"older regular intersection with exponent $1$. The result now follows from Theorem~\ref{thm:abstract convergence rate}.
\end{proof}

 The following result includes \cite[Th.~4.4]{Bor_Li_Yao} and \cite[Th.~3.12]{BB} a special cases.
\begin{corollary}[Cyclic projection algorithm with H\"older regularity]
 Let $J=\{1,2,\dots,m\}$ and let $\{C_j\}_{j\in J}$ a collection of closed convex subsets of a Hilbert space $H$ with non-empty intersection.  Given $x^0\in H$, set
  $$x^{t+1}={\rm P}_{C_{j}}(x^t)\text{ where } j=t~{\rm mod}~m.$$
Suppose that $\{C_j\}_{j\in J}$ has a bounded H\"older regular intersection. Then $x^t \to \bar x\in \cap_{j\in J}C_j\neq \emptyset$ at least with a sublinear rate $O(t^{-\rho})$ for some $\rho>0$.
 In particular, if the collection $\{C_j\}_{j\in J}$ is bounded H\"older regular with uniform exponent $\gamma\in(0,1]$ there exist $M>0$ and $r\in[0,1)$ such that
\[
\|x^{t}-\bar x\| \le \begin{cases}
                       M t^{-\frac{\gamma}{2(1-\gamma)}}, & \gamma \in (0,1), \\
                       M \, r^t,  & \gamma=1.
                     \end{cases}
\]
\end{corollary}
\begin{proof}
 First note that the projection operator over a closed convex set is $1/2$-averaged. Now, for each $j\in J$, $C_j={\rm Fix\,P}_{C_j}$, and hence
  $$d(x,C_j)=d(x,{\rm Fix\,P}_{C_j})=\|x-{\rm P}_{C_j}x\|\quad\forall x\in H.$$
 That is, for each $j\in J$, the projection operator ${\rm P}_{C_j}$ is bounded H\"older regular with uniform exponent $1$. The result follows from Theorem~\ref{thm:abstract convergence rate}.
\end{proof}

We now turn our attention to \emph{variational inequalities} \cite[Ch.~25.5]{BC2011}. Let $f:H\to(-\infty,+\infty]$ be a proper lower semi-continuous (l.s.c.\@) convex function and let $F: H \mapsto H$ be \emph{$\beta$-cocoercive}, that is,
\[
\langle F(x)-F(y),x-y\rangle \ge \beta \|F(x)-F(y)\|^2 \quad \forall x,y \in H.
\]
 The \emph{generalized variational inequality problem}, denoted ${\rm VIP}(F,f)$, is:
\begin{equation}\label{eq:VIP}
 \text{Find~}x^* \in H\text{~such that~}f(x)-f(x^*)+\langle F(x^*), x-x^*\rangle \ge 0,
\end{equation}
and the set of \emph{solutions} to \eqref{eq:VIP} is denoted ${\rm Sol}({\rm VIP}(F,f))$.

The \emph{forward-backward splitting method} is often employed to solve $VIP(F,f)$ (see, for instance, \cite[Prop.~25.18]{BC2011}) and generates a sequence $\{x^t\}$ according to
 \begin{equation}
 x^{t+1}=x^t+\lambda_t \big(R_{\gamma \partial f}(x^t-\gamma F(x^t))-x^t\big) \quad\forall t \in \mathbb{N}.
 \end{equation}
where $R_{T_0}:=(I+T_0)^{-1}$ denotes the \emph{resolvent} of an operator $T_0$. Here, we note that $\partial f$ is a maximal monotone operator and so, its resolvent is single-valued.

\begin{corollary}[Forward-backward splitting method for variational inequality problem] \label{cor:complexity_FB}
Let $f:H\to(-\infty,+\infty]$ be a proper l.s.c.\@ convex function, let $F:H \to H$ be a $\beta$-cocoercive operator with $\beta>0$, let $\gamma\in(0,2\beta)$, and let $(\lambda_t)_{t\in\mathbb{N}}$ be a sequence of real numbers with $\sigma_0:=\inf_{t\in\mathbb{N}}\{\lambda_t(1-\lambda_t)\}\in (0,1)$. Suppose ${\rm Sol}({\rm VIP}(F,f))\neq\emptyset$. Let $x^0 \in H$ and set
\begin{equation}\label{FB_VIP}
 x^{t+1}= (1-\lambda_t)\,x^t+ \lambda_t R_{\gamma \partial f}(x^t-\gamma F(x^t))\quad\forall t \in \mathbb{N}.
 \end{equation}
 If $T_2:=  R_{\gamma \partial f}(I-\gamma F)$ is bounded H\"older regular then $x^t \to \bar x \in {\rm Sol}({\rm VIP}(F,f))$ at least with a sublinear rate $O(t^{-\rho})$ for some $\rho>0$. In particular, if $T$ is bounded H\"older regular with uniform exponent $\gamma\in(0,1]$ then there exist $M>0$ and $r\in[0,1)$ such that
\[
\|x^{t}-\bar x\| \le \begin{cases}
                       M t^{-\frac{\gamma}{2(1-\gamma)}}, & \gamma \in (0,1), \\
                       M \, r^t,  & \gamma=1.
                     \end{cases}
\]
\end{corollary}
\begin{proof}
 Since $F$ is $\beta$-cocoercive, \cite[Prop.~4.33]{BC2011} shows that $I-\gamma F$ is $\gamma/(2\beta)$-averaged. The operator $R_{\gamma \partial f}$ is $\frac{1}{2}$-averaged, as the resolvent of a maximally monotone operator. By \cite[Prop.~4.32]{BC2011}, $T_2$ is $\frac{2}{3}$-averaged. Applying Corollary~\ref{cor:KM iteration} with $T=T_2$, the claimed convergence rate to a point $\bar{x}\in{\rm Fix\ }T_2$ thus follows.

 It remains to show that ${\rm Fix\ }T_2={\rm Sol}({\rm VIP}(F,f))$. Indeed, $x\in{\rm Fix\ }T_2$ if and only if
 $$ x=(I+\gamma \partial f)^{-1}\big((I-\gamma F)(x)\big) \iff x-\gamma F(x)\in x+\gamma \partial f(x) \iff 0\in (\partial f+F)(x).$$
The claimed result follows.
\end{proof}

\begin{remark}[Semi-algebraic forward-backward splitting]
 Corollary~\ref{cor:complexity_FB} applies, in particular, when $f$ and $F$ are semi-algebraic functions and $H=\mathbb{R}^n$. To see this, first note that $\partial f$ is semi-algebraic as the sub-differential of the semi-algebraic function \cite[p.~6]{drusvyatskiy}. The graph of the resolvent to $\gamma\partial f$ may then be expressed as
  $${\rm gph\ } R_{\gamma\partial f} = \{\left(z,(I+\gamma \partial f)^{-1}(z)\right):z\in\mathbb{R}^n\} = \{\left((I+\gamma \partial f)(z), z\right) : z\in\mathbb{R}^n\},$$
 which is a semi-algebraic set since $(I+\gamma\partial f)$ is a semi-algebraic function by Fact~\ref{fact:6}, and hence the resolvent is a semi-algebraic mapping. A further application of Fact~\ref{fact:6}, shows that $T_2$ is semi-algebraic, as the composition of semi-algebraic maps.

 We may therefore define two continuous semi-algebraic functions $\phi(x):=\di(x,{\rm Fix }T_2)$ and $\varphi(x):=\|x-Tx\|$. Since $\phi^{-1}(0)={\rm Fix }T_2 = \varphi^{-1}(0)$, (P6) of Fact~\ref{fact:6} yields, for any compact semi-algebraic set $K$, the existence of constants $c>0$ and $\tau\in(0,1]$ such that
  $$\|x-T_2x\| \leq c \di^\tau (x,{\rm Fix }T_2)\quad\forall x\in K,$$
 or, in other words, the bounded H\"older regularity of $T_2$.
\end{remark}

\section{The rate of convergence of DR algorithms for convex feasibility problems}\label{sec:DR}
 We now specialize our convergence results to the classical DR algorithm and its variants in the setting of convex feasibility problems. In doing so, a convergence rate is obtained under the H\"{o}lder regularity condition. Recall that the basic \emph{Douglas--Rachford algorithm} for two set feasibility problems can be stated as follows:
\vspace{-0.2cm}

\begin{center}
\begin{algorithm}[H] \label{alg:DR}
 \caption{Basic Douglas--Rachford algorithm}
 \KwData{Two closed and convex sets $C,D\subseteq H$}
 Choose an initial point $x^0\in H$\;
 \For{$t=0,1,2,3,\dots$}{
  Set:
  \begin{equation}\label{scheme}
   \left\{
    \begin{aligned}
      &y^{t+1}:= {\rm P}_C(x^t), \\
      &z^{t+1}:= {\rm P}_D(2y^{t+1}-x^t),\\
      &x^{t+1}:=x^t+(z^{t+1}- y^{t+1}).
    \end{aligned}
   \right.
  \end{equation}
  }
\end{algorithm}
\end{center}
Direct verification shows that the relationship between consecutive terms in the sequence $(x^t)$ of \eqref{scheme} can be described in terms of the firmly nonexpansive \emph{(two-set) Douglas--Rachford operator} which is of the form
  \begin{equation}\label{eq:two set DR}
   T_{C,D}=\frac{1}{2}\left(I+R_DR_C\right),
  \end{equation}
where $I$ is the identity mapping and $R_C:=2P_C-I$ is the \emph{reflection operator} with respect to the set $C$ (`reflect-reflect-average').

We shall also consider the abstraction given by Algortihm~\ref{alg:many set DR} which chooses two constraint sets from some finite collection at each iteration. Note that iterations \eqref{scheme} and \eqref{scheme many} have the same structure.

The motivation for studying Algorithm~\ref{alg:many set DR} is that, beyond Algorithm~\ref{alg:DR}, it include two further DR-type schemes from the literature.
The first scheme is the \emph{cyclic DR algorithm}  and is generated according to:
\begin{equation*}
u^{t+1} = (T_{C_m,C_1}\,T_{C_{m-1},C_m}\,\ldots\, T_{C_2,C_3}\,T_{C_1,C_2})(u^t) \quad\forall t \in \mathbb{N},
\end{equation*}
which corresponds the Algorithm~\ref{alg:many set DR} with $u^t=x^{mt}$, $t \in \NN$, $p=m$ and $\Omega_j=(j,j+1)$, $j=1,\ldots,m-1$, and $\Omega_m=(m,1)$.
The second scheme the \emph{cyclically anchored DR algorithm} and is generated according to:
\begin{equation*}
u^{t+1} = (T_{C_1,C_m}\, \ldots\, T_{C_1,C_3}\,T_{C_1,C_2})(u^t)\quad\forall t \in \mathbb{N},
\end{equation*}
which corresponds the Algorithm~\ref{alg:many set DR} with $u^t=x^{(m-1)t}$, $t \in \NN$, $p=m-1$ and $\Omega_j=(1,j+1)$, $j=1,\ldots,m-1$.
 The following lemma shows that underlying operators both these methods are also averaged.
\begin{lemma}[Compositions of DR operators] \label{lemma:averaged}
Let $p$ be a positive integer. The composition of $p$ Douglas--Rachford operators is $\frac{p}{p+1}$--averaged.
\end{lemma}
\begin{proof}
 The two-set Douglas--Rachford operator of \eqref{eq:two set DR} is firmly nonexpansive, and hence $1/2$-averaged. The result follows by \cite[Prop.~4.32]{BC2011}.
\end{proof}

\begin{center}
\begin{algorithm}[H]
 \caption{A multiple-sets Douglas--Rachford algorithm}\label{alg:many set DR}
 \KwData{A family of $m$ closed and convex sets $C_1,C_2,\dots,C_m\subseteq H$}
 Choose a list of $2$-tuples $\Omega_1,\dots,\Omega_p \in\{(i_1,i_2):i_1,i_2=1,2,\dots,m\text{ and }i_1\neq i_2\}$ with $\cup_{j=1}^p \Omega_j =\{1,\ldots,m\}$\;
 Define $\Omega_0=\Omega_m$\;
 Choose an initial point $x^0\in H$\;
 \For{$t=0,1,2,3,\dots$}{
  Set the indices $(i_1,i_2):=\Omega_{t'}$ where $t'=t+1{\rm~mod~}s$\;
  Set
  \begin{equation}\label{scheme many}
   \left\{
    \begin{aligned}
      &y^{t+1}:= {\rm P}_{C_{i_1}}(x^t), \\
      &z^{t+1}:= {\rm P}_{C_{i_2}}(2y^{t+1}-x^t),\\
      &x^{t+1}:=x^t+(z^{t+1}- y^{t+1}).
    \end{aligned}
   \right.
  \end{equation}
  }
\end{algorithm}
\end{center}

\begin{corollary}[Convergence rate for the multiple-sets DR algorithm] \label{TheMTR:6}
Let $C_1,C_2,\dots,C_m$ be closed convex sets in a Hilbert space $H$ with non-empty intersection.
Let $\{\Omega_j\}_{j=1}^p$ and $\{(y^t,z^t,x^t)\}$ be generated by the multiple-sets Douglas--Rachford algorithm \eqref{scheme}. Suppose that:
 \begin{enumerate}[(a)]
  \item For each $j\in\{1,\dots,p\}$, the operator $T_{\Omega_j}$ is bounded H\"{o}lder regular.
  \item The collection $\{{\rm Fix}\,T_{\Omega_j}\}_{j=1}^p$ has a bounded H\"older regular intersection.
\end{enumerate}
Then $x^t \rightarrow \bar x\in\cap_{j=1}^p{\rm Fix\,}T_{\Omega_j}$ with at least a sublinear rate $O(t^{-\rho})$ for some $\rho>0$. In
particular, suppose we assume the stronger assumptions:
\begin{enumerate}[(a$'$)]
 \item  For each $j\in\{1,\dots,p\}$, the operator $T_{\Omega_j}$ is  bounded H\"{o}lder regular with uniform exponent $\gamma_{1,j}$.
 \item The collection $\{{\rm Fix}\,T_{\Omega_j}\}_{j=1}^p$ has a bounded H\"older regular intersection with uniform exponent $\gamma_2\in(0,1]$.
\end{enumerate}
Then there exist $M>0$ and $r \in (0,1)$ such that
\[
\|x^{t}-\bar x\| \le \left\{\begin{array}{ccl}
M t^{-\frac{\gamma}{2(1-\gamma)}} & \mbox{ if } & \gamma \in (0,1), \\
                                    M \, r^t & \mbox{ if } & \gamma=1.
                       \end{array}
\right.
\]
where $\gamma:=\gamma_1 \gamma_2$ where $\gamma_1:=\min\{\gamma_{1,j}\mid 1\le j\le s\}$.
\end{corollary}
\begin{proof}
Let $J=\{1,2,\dots,s\}$. For all $j\in J$, set  $T_j = T_{\Omega_j}$ and
$$w_{t,j}\equiv \begin{cases}
                  1 & j=t+1{\rm~mod~}s, \\
                  0 & \text{otherwise}. \\
                \end{cases}.$$
 Since $T_{\Omega_j}$ is firmly nonexpansive (that is, $1/2$-averaged), the conclusion follows immediately
from Theorem~\ref{thm:abstract convergence rate}.
\end{proof}

We next observe that  bounded H\"{o}lder regularity of the Douglas-Rachford operator $T_{\Omega_j}$ and
H\"older regular intersection of the collection $\{{\rm Fix}\,T_{\Omega_j}\}_{j=1}^p$
are automatically satisfied for the semi-algebraic convex case, and so, sublinear convergence analysis follows in this case without any further regularity conditions.
This follows from:

\begin{proposition}[Semi-algebraicity implies H{\"o}lder regularity \& sublinear convergence] \label{prop:T_DR hoelder}
Let $C_1$, $C_2$, $\dots,C_m$ be convex, semi-algebraic sets in $\mathbb{R}^n$ with non-empty intersection which can be described by polynomials (in the sense of \eqref{eq:semi alg set}) on $\mathbb{R}^n$ having  degree $d$.
Let $\{\Omega_j\}_{j=1}^p$ be a list of $2$-tuples with $\cup_{j=1}^p \Omega_j =\{1,\ldots,m\}$.
Then,  \begin{enumerate}[(a)]
  \item For each $j\in\{1,\dots,p\}$, the operator $T_{\Omega_j}$ is bounded H\"{o}lder regular.  Moreover, if $d=1$, then ${\rm T}_{\Omega_j}$ is bounded H\"older regular with uniform exponent~$1$.

  \item The collection $\{{\rm Fix}\,T_{\Omega_j}\}_{j=1}^p$ has a bounded H\"older regular intersection.
\end{enumerate}
In particular, let $\{(y^t,z^t,x^t)\}$ be generated by the multiple-sets Douglas--Rachford algorithm \eqref{scheme many}. Then there exists $\rho>0$ such that $x^t \rightarrow \bar x\in\cap_{j=1}^p{\rm Fix\,}T_{\Omega_j}$ with at least a sublinear rate $O(t^{-\rho})$.
\end{proposition}
\begin{proof}
 Fix any $j \in \{1,\ldots,p\}$. We first verify that the operator $T_{\Omega_j}$ is bounded H\"{o}lder regular. Without loss of
 generality, we assume that $\Omega_j=\{1,2\}$ and so, $T_{\Omega_j}=T_{C_1,C_2}$ where $T_{C_1,C_2}$ is the Douglas-Rachford operator
 for $C_1$ and $C_2$. Recall that
 for each $x \in \mathbb{R}^n$,
$T_{C_1,C_2}(x)-x=  {\rm P}_{C_2}({\rm R}_{C_1}(x))-{\rm P}_{C_1}(x)$. We now distinguish two cases depending on the value of the degree $d$ of the polynomials which
describes $C_j$, $j=1,2$.

\emph{Case 1 $(d>1)$:} We first observe that, for a closed convex semi-algebraic set $C \subseteq \mathbb{R}^n$, the projection mapping $x \mapsto {\rm P}_C(x)$
 is a semi-algebraic mapping. This implies that, for $i=1,2$, $x\mapsto {\rm P}_{C_i}(x)$ and $x\mapsto {\rm R}_{C_i}(x)=2{\rm P}_{C_i}(x)-x$ are all semi-algebraic mappings.
 Since the composition of semi-algebraic maps remains semi-algebraic ((P5) of Fact~\ref{fact:6}), we deduce that
$f \colon x \mapsto \|T_{C_1,C_2}x-x\|^2$ is a continuous semi-algebraic function. By {\rm (P4)} of Fact~\ref{fact:6}, ${\rm Fix\,}T_{C_1,C_2}=\{x \mid f(x)=0\}$ which is therefore a semi-algebraic set. By {\rm (P2)} of Fact~\ref{fact:6}, the function $\di(\cdot, {\rm Fix\,} T_{C_1,C_2})$ is semi-algebraic, and clearly
$\di(\cdot, {\rm Fix\,} T_{C_1,C_2})^{-1}(0) = f^{-1}(0)$.

By the {\L}ojasiewicz inequality for semi-algebraic functions ({\rm (P6)} of Fact~\ref{fact:6}), we see that for every $\rho>0$, one can find $\mu>0$ and $\gamma \in (0,1]$ such that
\[
\di(x, {\rm Fix\, T_{C_1,C_2}}) \le \mu \|x-T_{C_1,C_2}x\|^{\gamma} \quad\forall x \in  \mathbb{B}(0,\rho).
\]
So, the Douglas--Rachford operator $T_{C_1,C_2}$ is bounded H\"{o}lder regular in this case.

 \emph{Case 2 $(d=1)$:} In this case, both $C_1$ and $C_2$ are polyhedral, hence their projections ${\rm P}_{C_1}$ and ${\rm P}_{C_2}$ are piecewise affine mappings. Noting that composition of piecewise affine mappings remains piecewise affine \cite{Sontag}, we deduce that
$F \colon x \mapsto T_{C_1,C_2}(x)-x$ is continuous and piecewise affine.
 Then, Robinson's theorem on metric subregularity of piecewise affine mappings \cite{Robinson} implies that for all $a \in \mathbb{R}^n$, there exist $\mu>0$, $\epsilon>0$  such that
\[
\di(x,{\rm Fix\,}T_{C_1,C_2})	=\di(x,F^{-1}(0)) \le \mu \|F(x)\|=\mu\|x-T_{C_1,C_2}(x)\|\quad\forall x \in \mathbb{B}(a,\epsilon).
\]
Then, a standard compactness argument shows that the Douglas--Rachford operator $T_{C_1,C_2}$ is bounded linear regular, that is, uniformly bounded H\"{o}lder regular with exponent $1$.

Next, we assert that the collection $\{{\rm Fix}\,T_{\Omega_j}\}_{j=1}^p$ has a bounded H\"older regular intersection. To see this, as in the proof of part (a), we can show that for each
$j=1,\ldots,p$, ${\rm Fix\,}T_{\Omega_j}$ is a semi-algebraic set. Then, their intersection
$\cap_{j=1}^p{\rm Fix\,}T_{\Omega_j}$ is also a semi-algebraic set. Thus, $\psi(x)=\di(x, \cap_{j=1}^p{\rm Fix\,}T_{\Omega_j})$ and
$\displaystyle \phi(x)=\max_{1 \le j \le p}\di(x, {\rm Fix\,}T_{\Omega_j})$ are semi-algebraic functions. It is easy to see that $\phi^{-1}(0)=\psi^{-1}(0)$ and hence  the {\L}ojasiewicz inequality for semi-algebraic functions ({\rm (P6)} of Fact~\ref{fact:6})
implies that the collection $\{{\rm Fix}\,T_{\Omega_j}\}_{j=1}^p$ has a bounded H\"older regular intersection.

The final conclusion follows by Theorem~\ref{thm:abstract convergence rate}.
\end{proof}

Next, we establish the convergence rate for DR algorithm assuming bounded  H\"{o}lder regularity of the Douglas--Rachford operator $T_{C,D}$.
\begin{corollary}[Convergence rate for the DR algorithm] \label{cor:DR holder}
Let $C,D$ be two closed convex sets in a Hilbert space $H$ with $C\cap D \neq \emptyset$, and let $T_{C,D}$ be the Douglas--Rachford operator.
Let $\{(y^t,z^t,x^t)\}$ be generated by the Douglas--Rachford algorithm \eqref{scheme}. Suppose that $T_{C,D}$ is bounded  H\"{o}lder regular. Then  $x^t \rightarrow \bar x\in {\rm Fix\,}T_{C,D}$ with at least a sublinear rate $O(t^{-\rho})$ for some $\rho>0$. In particular, if $T_{C,D}$ is bounded H\"older regular with uniform exponent $\gamma \in (0,1]$ then there exist $M>0$ and $r \in (0,1)$ such that
\[
\|x^{t}-\bar x\| \le \left\{\begin{array}{ccl}
M t^{-\frac{\gamma}{2(1-\gamma)}} & \mbox{ if } & \gamma \in (0,1), \\
                                    M \, r^t & \mbox{ if } & \gamma=1.
                       \end{array}
\right.
\]
\end{corollary}
\begin{proof}
Let $J=\{1\}$,  $T_1 = T_{C,D}$ and $w_{t,1}\equiv 1$. Note that $T_{C,D}$ is firmly nonexpansive (that is, $1/2$-averaged) and any collection containing only one set has H\"older regular intersection with exponent one. Then the conclusion follows
immediately from Theorem \ref{thm:abstract convergence rate}.
\end{proof}

Similar to Proposition \ref{prop:T_DR hoelder}, if $C$ and $D$ are basic convex semi-algebraic sets, then DR algorithm exhibits at least a sublinear convergence rate.
\begin{remark}[Linear convergence of the DR algorithm for convex feasibility problems]
We note that if $H=\mathbb{R}^n$ and $C,D$ are convex sets with ${\rm ri}C \cap {\rm ri D} \neq \emptyset$, then $T_{C,D}$ is bounded linear regular,
that is, bounded H\"{o}lder regular with uniform exponent $1$. In this case, the Douglas--Rachford algorithm converges linearly  as was previously established in \cite{BNP}. Further, if $C$ and $D$ are both subspaces such that $C+D$ is closed (as is automatic in finite-dimensions), then $T_{C,D}$ is also bounded linear regular, and so, the DR algorithm converges linearly in this case as well. This was  established in \cite{BBNP}. It should be noted that \cite{BBNP} deduced the stronger result that the linear convergence rate is exactly the cosine of the Friedrichs angle.
We also remark that linear convergence of DR algorithm under the regularity condition ${\rm ri}C \cap {\rm ri D} \neq \emptyset$ may alternatively be deduced as a consequence
of the recently established local linear convergence of nonconvex DR algorithms \cite{Luke2,Phan} by specializing $C$ and $D$  to be convex sets.
\end{remark}

{To conclude this section, we consider a regularization of the DR algorithm \cite{Luke3} which converges even when the target intersection is empty where the sequence is generated by $x^{t+1}=T_{R}(x^t)$
where $T_{R}:=\beta{\rm P_C}+(1-\beta)T_{C,D}$ and $\beta \in (0,1)$.  When the target intersection $C \cap D$ is empty, this gives an example of a useful algorithm in which the two operators of interest, $P_C$ and $T_{C,D}$, have no common fixed point but can still be analyzed within our framework.
\begin{proposition}[Convergence rate of regularized DR algorithm]\label{prop:RAAR}
Let $C,D$ be two basic convex semi-algebraic sets and let $T_{C,D}$ be the Douglas--Rachford operator.  Let $\{x^t\}$ be generated by $x^{t+1}:=T_R(x^t)$
where $T_R=\beta{\rm P_C}+(1-\beta)T_{C,D}$ and  $\beta \in(0,1)$.  Then  ${\rm P}_{D}(x^t) \rightarrow \bar x_1 \in D$ and ${\rm P}_C ({\rm P}_{D}(x^t)) \rightarrow \bar x_2 \in C$
both with at least a sublinear rate $O(t^{-\rho})$ for some $\rho>0$ and $\|\bar x_1-\bar x_2\|= \di(C,D)$.
In particular, if $C,D$ are polyhedral, then  there exist $M>0$ and $r \in (0,1)$ such that
\[
\max\{\|{\rm P}_D(x^{t})-\bar x_1\|, \|{\rm P}_C ({\rm P}_{D}(x^t))-\bar x_2\|\} \le M\, r^t.
\]
\end{proposition}
\begin{proof}
As $C,D$ be two basic convex semi-algebraic sets, $D-C$ is a closed set \cite[Lemma 4.7]{Bor_Li_Yao}. Let $g=P_{D-C}(0)$,
$E=C \cap (D-g)$ and $F=(C+g) \cap D$. Then, \cite[Lemma 2.1]{Luke3} shows that ${\rm Fix}T_R=F-\frac{\beta}{1-\beta}\, g$, $P_D({\rm Fix}T_R) \subseteq F$ and $P_C(P_D({\rm Fix}T_R)) \subseteq E$.  We first show that $T_R$ is bounded H\"older regular, and $T_R$ is bounded H\"older regular with uniform exponent $1$ if $C,D$ are polyhedral.
To see this, we use the same argument as in Proposition~\ref{prop:T_DR hoelder} and  it suffices to establish that $T_R-I$ is a semi-algebraic (resp. piecewise affine) map when $C$ and $D$ are semi-algebraic (resp. polyhedral). For the sake of avoiding repetition, we only show the latter. Observe that
  $$ T_R-I = \beta ({\rm P_{C}}-I)+(1-\beta)({\rm P_{D}R_{C}}-{\rm P_{C}}). $$
 Thus $T_R-I$ can be represented in terms of linear combinations and compositions of continuous semi-algebraic (resp. continuous piecewise affine) operators; more precisely, the projectors onto $C$ and $D$. Since projectors of convex semi-algebraic (resp. polyhedral) sets are semi-algebraic (resp. piecewise affine) and continuous, it follows that $T_R-I$ is semi-algebraic (resp. piecewise affine) and continuous.
It then follows from Theorem \ref{thm:abstract convergence rate} that $x^{t} \rightarrow \bar x \in {\rm Fix}T_R$ with at least a sublinear rate $O(t^{-\rho})$ for some $\rho>0$. From the non-expansive property of projection mapping,
we see that ${\rm P}_{D}(x^t) \rightarrow \bar x_1 \in P_D({\rm Fix}T_R) \subseteq F$ and ${\rm P}_C ({\rm P}_{D}(x^t)) \rightarrow \bar x_2 =P_C(\bar x_1) \subseteq E$
both with at least a sublinear rate $O(t^{-\rho})$. From the definitions of $E$, $F$ and $g$, it follows that $\|\bar x_1-\bar x_2\|=\di(C,D)$. The assertion for the linear convergence in the case where $C$ and $D$ are polyhedral also follows
from Theorem \ref{thm:abstract convergence rate}.
\end{proof}}

\section{The rate of convergence of the damped DR algorithm}\label{sec:damped DR}
 We now investigate a variant of Algorithm~1 which we refer to as the \emph{damped Douglas--Rachford algorithm}. To proceed, let $\eta>0$, let $A$ be a closed convex set in $H$, and define the operator ${\rm P}_A^\eta$ by
\[
{\rm P}_A^{\eta}= \left(\frac{1}{2\eta+1}I+\frac{2\eta}{2\eta+1} {\rm P}_A \right),
\]
where $I$ denotes the identity operator on $H$. The operator ${\rm P}_A^\eta$ can be considered as a relaxation of the projection mapping. Further, a direct verification shows that
\[
\lim_{\eta \rightarrow \infty} {\rm P}_A^{\eta}(x)={\rm P}_A(x) \quad\forall x \in H,
\]
in norm, and
\begin{equation}\label{eq:96}
{\rm P}_A^{\eta}(x)={\rm prox}_{\eta \di_A^2}(x)=\argmin_{y \in H} \left\{\di_A^2(y)+\frac{1}{2\eta}\|y - x\|^2\right\}\quad\forall x \in H,
\end{equation}
where  ${\rm prox}_f$ denotes the \emph{proximity operator} of the function $f$. The damped variant can be stated as follows:

\begin{center}
\begin{algorithm}[H]\label{alg:damped DR}
 \caption{Damped Douglas--Rachford algorithm}
 \KwData{Two closed sets $C,D\subseteq H$}
 Choose $\eta>0$ and $\lambda \in (0,2]$;\\
 Choose an initial point $x^0\in H$\;
 \For{$t=0,1,2,3,\dots$}{
  Choose $\lambda_t \in (0,2]$ with $\lambda_t\geq\lambda$ and set:
  \begin{equation}\label{scheme-damped}
	\left\{
	\begin{aligned}
	&y^{t+1}:= {\rm P}_C^{\eta}(x^t), \\
	&z^{t+1}:= {\rm P}_D^{\eta}(2y^{t+1}-x^t),\\
	&x^{t+1}:=x^t+\lambda_t(z^{t+1}- y^{t+1}).
	\end{aligned}
	\right.
  \end{equation}
  }
\end{algorithm}
\end{center}

\begin{remark}
In the more general setting in which  $P_C^{\eta}$ and $P_D^{\eta}$ are, respectively, replaced by ${\rm prox}_{\eta f}$ and ${\rm prox}_{\eta g}$ for $f$ and $g$ proper lower semicontinuous convex functions,  Algorithm~\ref{alg:damped DR} can be found, for instance, in \cite[Cor.~27.4]{BC2011}, \cite[Alg.~1.8]{Luke3} and \cite{Combettes1}. Convergence of this algorithm (without an explicit estimate of the convergence rate) has been established in \cite[Cor.~1.11]{Luke3} and
\cite[Cor.~5.2]{Combettes1}.
 We also note that a similar relaxation of the Douglas--Rachford algorithm for lattice cone constraints has been proposed and analyzed in \cite{BorSimTam}.
\end{remark}

 Whilst it is possible to analyze the damped Douglas--Rachford algorithm within the quasi-cyclic framework, we learn more by proving the following result  directly.

\begin{theorem}[Convergence Rate for the  damped Douglas--Rachford algorithm] \label{th:01}
Let $C,D$ be two closed and convex sets in a Hilbert space $H$ with $C \cap D \neq \emptyset$.  Let $\lambda:=\inf_{t \in \mathbb{N}} \lambda_t >0$ with
$\lambda_t \in (0,2]$ and let  $\{(y^t,z^t,x^t)\}$ be generated by the damped Douglas--Rachford algorithm \eqref{scheme-damped}. Suppose that the pair of sets $\{C,D\}$ has a bounded H\"older regular intersection.
Then  $x^{t} \rightarrow \bar x \in C\cap D$ with at least a sublinear rate $O(t^{-\rho})$ for some $\rho>0$. Furthermore, if the pair $\{C,D\}$ has a bounded H\"older regular intersection with uniform exponent $\gamma\in(0,1]$ then there exist $M>0$ and $r \in (0,1)$ such that
\[
\|x^{t}-\bar x\| \le \left\{\begin{array}{ccl}
M t^{-\frac{\gamma}{2(1-\gamma)}} & \mbox{ if } & \gamma \in (0,1) \\
                                    M \, r^t & \mbox{ if } & \gamma=1.
                       \end{array}
\right.
\]
\end{theorem}
\begin{proof}
\emph{Step~1 (A Fej\'{e}r
monotonicity type inequality for $x^t$):} Let $x^* \in C \cap D$.
We first show that
\begin{equation}\label{claim1}
2 \eta\lambda \big(\di^2(y^{t+1},C)+\di^2(z^{t+1},D)\big) \le \|x^{t}-x^*\|^2-\|x^{t+1}-x^*\|^2.
\end{equation}
To see this, note that, for any closed and convex set $A$, $\di^2(\cdot,A)$ is a differentiable convex function satisfying
$\nabla (\di^2)(x,A)=2(x-{\rm P}_{A}(x))$ which is $2$-Lipschitz. Using the convex subgradient inequality, we have
 \begin{eqnarray} \label{eq:003}
& & 2 \eta\lambda\big(\di^2(y^{t+1},C)+\di^2(z^{t+1},D)\big) \nonumber \\
& \leq & 2 \eta\lambda_t\big(\di^2(y^{t+1},C)+\di^2(z^{t+1},D)\big) \nonumber \\
& = & 2 \eta\lambda_t \big(\di^2(y^{t+1},C)-\di^2(x^*,C)+\di^2(z^{t+1},D)-\di^2(x^*,D)\big) \nonumber \\
& \le &  4 \eta\lambda_t\big(\langle y^{t+1}-{\rm P_C}(y^{t+1}), y^{t+1}-x^*\rangle+\langle z^{t+1}-{\rm P}_D(z^{t+1}), z^{t+1}-x^*\rangle\big) \nonumber\\
& = & 4 \eta\lambda_t\big(\langle y^{t+1}-{\rm P_C}(y^{t+1}), y^{t+1}-x^*\rangle+\langle z^{t+1}-{\rm P}_D(z^{t+1}), z^{t+1}-y^{t+1}\rangle \nonumber\\
& & \ \ \ +\langle z^{t+1}-{\rm P}_D(z^{t+1}), y^{t+1}-x^*\rangle\big) \nonumber \\
& = & 4 \eta\lambda_t\big(\langle y^{t+1}-{\rm P_C}(y^{t+1})+z^{t+1}-{\rm P}_D(z^{t+1}), y^{t+1}-x^*\rangle \nonumber \\
& & +\langle z^{t+1}-{\rm P}_D(z^{t+1}), z^{t+1}-y^{t+1}\rangle )\nonumber \\
& = & 4 \eta\big(\lambda_t\langle y^{t+1}-{\rm P_C}(y^{t+1})+z^{t+1}-{\rm P}_D(z^{t+1}), y^{t+1}-x^*\rangle \label{eq:first} \\
& & +\langle z^{t+1}-{\rm P}_D(z^{t+1}), x^{t+1}-x^{t}\rangle ), \nonumber
\end{eqnarray}
where the last equality follows from the last relation in \eqref{scheme-damped}. Now using (\ref{eq:96}), we see that
 \begin{align*}
0 &=\nabla (\di^2(\cdot,C)+ \frac{1}{2\eta}\|\cdot-x^t\|^2)(y^{t+1})= 2\left( y^{t+1}-{\rm P_C}(y^{t+1})\right)+\frac{1}{\eta}(y^{t+1}-x^t),
\end{align*}
and similarly 
\begin{align*}
0 &=\nabla (\di^2(\cdot,D)+ \frac{1}{2\eta}\|\cdot-(2y^{t+1}-x^t)\|^2)(z^{t+1})  \\
&= 2\left( z^{t+1}-{\rm P_D}(z^{t+1}) \right) +\frac{1}{\eta}(z^{t+1}-2y^{t+1}+x^t).
\end{align*}
Summing these two equalities and multiplying by $\lambda_t$ yields
\[
\lambda_t\left( y^{t+1}-{\rm P_C}(y^{t+1})+z^{t+1}-{\rm P}_D(z^{t+1})\right)=-\frac{\lambda_t}{2\eta} (z^{t+1}-y^{t+1})= -\frac{1}{2\eta} (x^{t+1}-x^{t}).
\]
Note also that
 $$x^t+z^{t+1}-y^{t+1}=x^{t+1}+(1-\lambda_t)(z^{t+1}-y^{t+1})=x^{t+1}+\frac{1-\lambda_t}{\lambda_t}(x^{t+1}-x^t).$$
Substituting the last two equations into \eqref{eq:003} gives
\begin{align*}
& 2\eta\lambda\big(\di^2(y^{t+1},C)+\di^2(z^{t+1},D)\big) \\
& \quad\le   4 \eta \langle z^{t+1}-{\rm P}_D(z^{t+1})- \frac{1}{2\eta} (y^{t+1}-x^*) , x^{t+1}-x^{t} \rangle \\
& \quad=  4 \eta\langle-\frac{1}{2\eta}(z^{t+1}-2y^{t+1}+x^t)- \frac{1}{2\eta} (y^{t+1}-x^*) , x^{t+1}-x^{t} \rangle \\
& \quad=  -2 \langle z^{t+1}-y^{t+1}+x^t-x^*, x^{t+1}-x^{t} \rangle \\
& \quad=  -2 \langle x^{t+1}-x^*, x^{t+1}-x^{t} \rangle -2\frac{1-\lambda_t}{\lambda_t}\|x^{t+1}-x^t\|^2\\
& \quad= \left( \|x^{t}-x^*\|^2-\|x^{t+1}-x^*\|^2-\|x^{t+1}-x^{t}\|^2\right) -2\frac{1-\lambda_t}{\lambda_t}\|x^{t+1}-x^t\|^2 \\
& \quad=  \|x^{t}-x^*\|^2-\|x^{t+1}-x^*\|^2 - \frac{2-\lambda_t}{\lambda_t}\|x^{t+1}-x^t\|^2.
\end{align*}

\noindent\emph{Step~2 (establishing a recurrence for $\di^2 (x^t,{C \cap D})$):} First note that
\[
y^{t+1}= {\rm P}_C^{\eta}(x^t)= \frac{1}{2\eta+1}x^t +\frac{2\eta}{2\eta+1} {\rm P}_C(x^t).
\]
This shows that $y^{t+1}$ lies in the line segment between $x^t$ and its projection onto $C$. So, $P_C(y^{t+1})={\rm P}_C(x^t)$ and hence,
\begin{align*}
\di^2(y^{t+1},C) = \|y^{t+1}- {\rm P}_C(x^t)\|^2
&= \left(\frac{1}{2\eta+1}\right)^{2} \|{\rm P}_C(x^t)-x^t\|^2 \\
&= \left(\frac{1}{2\eta+1}\right)^{2}\,\di^2 (x^t,C).
\end{align*}
Similarly, as
\[
z^{t+1}= {\rm P}_D^{\eta}(2y^{t+1}-x^t)=\frac{1}{2\eta+1}(2y^{t+1}-x^t) +\frac{2\eta}{2\eta+1} {\rm P}_D(2y^{t+1}-x^t),
\]
 the point $z^{t+1}$ lies in the line segment between $2y^{t+1}-x^t$ and its projection onto $D$. Thus $P_D(z^{t+1})={\rm P}_D(2y^{t+1}-x^t)$ and so,
\begin{align*}
\di^2(z^{t+1},D) &= \|z^{t+1}-{\rm P}_D(2y^{t+1}-x^t)\|^2 \\
&= \left(\frac{1}{2\eta+1}\right)^{2} \|{\rm P}_D(2y^{t+1}-x^t)-(2y^{t+1}-x^t)\|^2 \\
&= \left(\frac{1}{2\eta+1}\right)^{2} \di^2 (2y^{t+1}-x^t,D).
\end{align*}
Now, using the non-expansiveness of $\di(\cdot,D)$, \ we have
\begin{eqnarray*}
\di^2(x^{t},D) & \le &  \bigg(\|x^{t}-(2y^{t+1}-x^t)\|+ \di(2y^{t+1}-x^t,D)\bigg)^2 \\
& = & \bigg(2\|x^{t}-y^{t+1}\|+ \di(2y^{t+1}-x^t,D)\bigg)^2 \\
& = & \bigg(\frac{4\eta}{2\eta+1}\di(x^t,C) + \di(2y^{t+1}-x^t,D)\bigg)^2 \\
& \le &  c \, \bigg(\di^2(x^t,C) + \di^2(2y^{t+1}-x^t,D)\bigg),
\end{eqnarray*}
where $c:=2(\max\{\frac{4\eta}{2\eta+1},1\})^2$, and where  the last inequality above follows from the following elementary inequalities: for all $\alpha,x,y \in \mathbb{R}_+$,
 $$\alpha x+y\leq \max\{\alpha,1\}(x+y),\quad (x+y)^2 \le 2(x^2+y^2).$$
Therefore, we have
\[
\di^2(2y^{t+1}-x^t,D) \ge c^{-1}\di^2(x^{t},D)- \di^2(x^t,C).
\]
So, using \eqref{claim1}, we have
\begin{eqnarray*}
\|x^{t}-x^*\|^2-\|x^{t+1}-x^*\|^2 & \ge & 2\eta\lambda\big(\di^2(y^{t+1},C)+\di^2(z^{t+1},D)\big) \\
& = & 2\eta\lambda\left(\frac{1}{2\eta+1}\right)^{2}\bigg( \di^2 (x^t,C)+ \di^2 (2y^{t+1}-x^t,D)\bigg). \\
\end{eqnarray*}
Note that
\begin{align*}
\di^2 (x^t,C)+ \di^2 (2y^{t+1}-x^t,D) 
&\ge   \di^2 (x^t,C)+ c^{-1}\di^2(x^{t},D)-  \di^2(x^t,C) \\
&= c^{-1}\di^2(x^{t},D)
\end{align*}
and
\[
\di^2 (x^t,C)+ \di^2 (2y^{t+1}-x^t,D) \ge \di^2 (x^t,C).
\]
It follows that
\begin{eqnarray}\label{eq:Recurrence}
\|x^{t}-x^*\|^2-\|x^{t+1}-x^*\|^2  &\ge &  2\eta \left(\frac{1}{2\eta+1}\right)^{2} c^{-1} \max\{\di^2 (x^t,C),\di^2 (x^t,D)\} \nonumber \\
&=&  2\eta \left(\frac{1}{2\eta+1}\right)^{2} c^{-1} (\max\{\di (x^t,C),\di (x^t,D)\})^2.
\end{eqnarray}
In particular, we see that the sequence $\{x^{t}\}$ is bounded and Fej\'{e}r
monotone with respect to $C \cap D$.
Thence,  letting $K$ be a bounded set containing $\{x^t\}$, by bounded H\"{o}lder regularity of $\{C,D\}$, there exists $\mu>0$ and $\gamma\in(0,1]$ such that
 $$\di(x,{C\cap D})\leq \mu\max\{\di(x,C),\di(x,D)\}^\gamma\quad\forall x\in K.$$
Thus there exists a $\delta>0$ such that
\[
\|x^{t}-x^*\|^2-\|x^{t+1}-x^*\|^2  \ge \delta\, \di^{2\theta} (x^t,{C \cap D})
\]
where  $\theta=\frac{1}{\gamma} \in [1,\infty)$. So, Fact \ref{FactPr:3} implies that ${\rm P}_{C \cap D} (x^t) \rightarrow \bar x$ for some $\bar x \in C \cap D$. Setting $x^*={\rm P}_{C \cap D} (x^t)$ in \eqref{eq:Recurrence} we therefore obtain
\[
\di^2 (x^{t+1},{C \cap D}) \le \di^2 (x^t,{C \cap D})-\delta \di^{2\theta} (x^t,{C \cap D}).
\]
Now, the conclusion follows by applying Proposition \ref{prop:convergence rate} with $\theta=1/\gamma$.
\end{proof}

\begin{remark}[DR versus damped DR]
 Note that Theorem~\ref{th:01} only requires H\"older regularity of the underlying collection of constraint sets,  rather than the damped DR operator explicitly. A careful examination of the proof of Theorem~\ref{th:01} shows that the inequality \eqref{eq:first} does not hold for the basic DR algorithm (which would require setting $\eta=+\infty$).
\end{remark}

\begin{remark}[Comments on linear convergence]
In the case when $0 \in {\rm sri}(C-D)$, where ${\rm sri}$ is the \emph{strong relative interior},
then the H\"{o}lder regularity result holds with exponent
$\gamma=1$ (see \cite{BB}). The preceding proposition therefore implies that the damped Douglas--Rachford method converges linearly in the case where
 $0 \in {\rm sri}(C-D)$.
\end{remark}

We next show  that an explicit sublinear convergence rate estimate can be achieved in the case where $H=\mathbb{R}^n$ and $C,D$ are convex basic semi-algebraic sets.

\begin{theorem}[Convergence rate for the damped DR algorithm with semi-algebraic sets] \label{TheMTR:3}
Let $C,D$ be two basic convex semi-algebraic sets in $\mathbb{R}^n$ with $C\cap D \neq \emptyset$, where $C,D$ are given by
\[
C:=\{x \in \mathbb{R}^n \mid g_i(x) \le 0, i=1,\ldots,m_1\} \mbox{ and } D:=\{x \in \mathbb{R}^n \mid h_j(x) \le 0, j=1,\ldots,m_2\}
\]
where $g_i,h_j$, $i=1,\ldots,m_1$, $j=1,\ldots,m_2$, are convex polynomials on $\mathbb{R}^n$ with degree at most $d$.
Let $\lambda:=\inf_{t \in \mathbb{N}} \lambda_t >0$ with $\lambda_t \in (0,2]$ and let  $\{(y^t,z^t,x^t)$ be generated by the damped Douglas--Rachford algorithm \eqref{scheme-damped}.
Then, $x^{t} \rightarrow \bar x \in C\cap D$. Moreover, there exist $M>0$ and $r \in (0,1)$ such that
\[
\|x^{t}-\bar x\| \le \left\{\begin{array}{ccl}
M t^{-\frac{\gamma}{2(1-\gamma)}} & \mbox{ if } & d>1 \\
                                    M \, r^t & \mbox{ if } & d=1.
                       \end{array}
\right.
\]
where $\gamma=[\min\left\{\frac{(2d-1)^n+1}{2},\, B(n-1)d^n\right\}]^{-1}$ and $\beta(n-1)$, is the central binomial coefficient with respect to $n-1$
which is given by $\binom{n-1}{[(n-1)/2]}$.
\end{theorem}
\begin{proof}
By Lemma~\ref{ThesumSet:1} with $\theta=1$, we see that for any compact set $K$,  there exists  $c> 0$ such that for all $x \in K$,
\begin{eqnarray*}
\di (x,C \cap D) & \le & c \bigg( \di(x, C)+\di(x,D) \bigg)^{\gamma}
 \le  2^{\gamma} c\, \max\{ \di(x,C),\di(x,D)\}^{\gamma}.
\end{eqnarray*}
where $\gamma=[\min\left\{\frac{(2d-1)^n+1}{2},\, B(n-1)d^n\right\}]^{-1}$. Note that $\gamma=1$ if $d=1$; while $\gamma \in (0,1)$ if $d>1$. The conclusion now follows from Theorem~\ref{th:01}.\end{proof}
\begin{remark}
{ Let $C,D$ be two basic convex semi-algebraic sets in $\mathbb{R}^n$ with $C\cap D \neq \emptyset$ and consider the associated convex feasibility problem: find $x^* \in C \cap D$.
As an easy consequence of Theorem \ref{TheMTR:3}, we see that a solution with $\epsilon$-tolerance of the convex feasibility problem, the number of iterations needed of the
damped DR algorithm is at worst $O(\frac{1}{\sqrt[\rho]{\epsilon}})$  where $\rho:=\frac{\gamma}{2(1-\gamma)}$ and $\gamma$ is a constant given in Theorem \ref{TheMTR:3} that can be explicitly
determined.}
\end{remark}

\section{Two Examples}\label{sec:examples}
 In this section we fully examine two concrete problems which illustrate the difficulty of establishing optimal rates. {In addition to illustrating our approach, these examples also give some further insight into the sharpness of our derived qualitative behavior.} We begin with an example consisting of two sets having an intersection which is bounded H\"older regular but not bounded linearly regular. In the special case where $n=1$, it has previously been examined in detail as part of \cite[Ex.~5.4]{Heinz_0}.

\begin{example}[Half-space and epigraphical set described by $\|x\|^d$]
Consider the sets  $$C=\{(x,r)\in \mathbb{R}^{n} \times \mathbb{R} \mid r \le 0\} \mbox{~and~}
D=\{(x,r) \in \mathbb{R}^{n} \times \mathbb{R} \mid r \ge \|(x_1,\ldots,x_n)\|^d\},$$ where $d>0$ is an even number. Clearly,
$C \cap D=\{0_{\mathbb{R}^{n+1}}\}$ and ${\rm ri}C \cap {\rm ri} D=\emptyset$. It can be directly verified that $\{C,D\}$ does not has a bounded linearly
regular intersection because, for $x_k:=(\frac{1}{k},0,\ldots,0)\in \mathbb{R}^n$ and $r_k:=\frac{1}{k^d}$,
\[
\di \big((x_k,r_k),C\cap D\big)=O\left(\frac{1}{k}\right) \mbox{ and } \max\{\di\big((x_k,r_k),C\big),\di\big((x_k,r_k),D\big)\}=\frac{1}{k^d}.
\]

Let $T_{C,D}$ be the Douglas--Rachford operator with
respect to the sets $C$ and $D$.
We will verify that $T_{C,D}$ is bounded  H\"{o}lder regular with exponent $\frac{1}{d}$.
Granting this,  by Corollary~\ref{TheMTR:6}, the sequence $(x^t,r^t)$ generated by the Douglas--Rachford algorithm converges to a point in ${\rm Fix\,}T_{C,D}=\{0_{\mathbb{R}^n}\} \times \mathbb{R}_+$ at least at the order of $t^{\frac{-1}{2(d-1)}}$, regardless of the chosen initial point

Firstly, on the route to showing bounded H\"{o}lder regularity,    it can be verified (see also \cite[Cor.~3.9]{bauschke2004finding}) that
\[
{\rm Fix\,}T_{C,D}=C \cap D+ N_{\overline{C-D}}(0)=\{0_{\mathbb{R}^n}\} \times \mathbb{R}_+,
\]
and so,
$$\di((x,r),{\rm Fix\,}T)=\left\{\begin{array}{ccc}
\|x\| & \mbox{ if } & r\ge 0 \\
\|(x,r)\| & \mbox{ if } & r< 0.
\end{array}\right.$$
Moreover, for all $(x,r) \in \mathbb{R}^{n} \times \mathbb{R}$,
\[
(x,r)-T_{C,D}(x,r)={\rm P}_{D}({\rm R}_{C}(x,r))-{\rm P}_C(x,r)= {\rm P}_{D}(x,-|r|) -(x, \min\{r,0\}).
\]
Note that, for any $(z,s) \in \mathbb{R}^{n} \times \mathbb{R}$, denote $(z^+,s^+)=P_D(z,s)$. Then we have
\[
s^+ =\|z^+\|^d \mbox{ and } (z^+ - z)+d \|z^+\|^{d-2} \big(\|z^{+}\|^d-s\big)z^+=0.
\]
Let $(a,\gamma)={\rm P}_{D}(x,-|r|)$. Then, $a=0_{\mathbb{R}^n}$ if and only if $x=0_{\mathbb{R}^n}$,
\[
a-x= -d \|a\|^{d-2}(\|a\|^d+|r|)a \mbox{ and } \gamma=\|a\|^{d}.
\]
It follows that
\begin{equation}\label{eq:pp}
a =\frac{1}{1+d\|a\|^{2d-2}+d\|a\|^{d-2}|r|}x.
\end{equation}
So,
\begin{eqnarray*}
 (x,r)-T_{C,D}(x,r) &= & (-d \|a\|^{d-2}(\|a\|^d+|r|)a, \|a\|^d-\min\{r,0\})\\
 &= & \begin{cases}
(-d \|a\|^{d-2}(\|a\|^d+r)a, \|a\|^d) & \mbox{ if }  r\ge 0 \\
(-d \|a\|^{d-2}(\|a\|^d-r)a, \|a\|^d-r) & \mbox{ if }  r< 0.
\end{cases}
\end{eqnarray*}
Let $K$ be any bounded set of $\mathbb{R}^{n+1}$ and consider any $(x,r) \in K$. By the nonexpansivity  of the projection mapping, $(a,\gamma)={\rm P}_{D}(x,-|r|)$ is also bounded for any $(x_1,x_2)\in K$. Let $M>0$ be such that $\|(a,\gamma)\| \le M$ and $\|(x,r)\| \le M$ for all $(x,r) \in K$.
 To verify the bounded H\"older regularity, we divide the discussion into two cases depending on the sign of $r$.

  \emph{Case 1 $(r \ge 0)$:}~As $d$ is  even, it follows that for all $(x,r) \in K$ with $x \neq 0_{\mathbb{R}^n}$
\begin{eqnarray*}
\frac{\|(x,r)-T_{C,D}(x,r)\|^2}{\|x\|^{2d}} 
& = & \frac{d^2 \|a\|^{2(d-1)}(\|a\|^d+r)^2+\|a\|^{2d}}{\|x\|^{2d}}\\
& \ge  & \frac{\|a\|^{2d}}{\|x\|^{2d}} 
 =  \bigg(\frac{1}{1+d\|a\|^{2d-2}+d\|a\|^{d-2}|r|}\bigg)^{2d} \\
 & \ge &  \bigg(\frac{1}{1+dM^{2d-2}+dM^{d-1}}\bigg)^{2d},
\end{eqnarray*}
where the equality follows from (\ref{eq:pp}).
This shows that, for all $(x,r) \in K$,
\[
\di((x,r),{\rm Fix\,}T_{C,D}) \le (1+dM^{2d-2}+dM^{d-1}) \|(x,r)-T_{C,D}(x,r)\|^{\frac{1}{d}}.
\]

 \emph{Case 2 $(r < 0)$:}
 As $d$ is  even, it follows that for all $(x,r) \in K \backslash\{0_{\mathbb{R}^{n+1}}\}$,
\begin{eqnarray*}
\frac{\|(x,r)-T_{C,D}(x,r)\|^2}{\|x\|^{2d}+r^{2d}} &=& \frac{(1+d^2 \|a\|^{2(d-1)})(\|a\|^{d}-r)^2}{\|x\|^{2d}+r^{2d}} \\
& \ge  & \frac{\|a\|^{2d}+r^2}{\|x\|^{2d}+r^{2d}} 
 \ge  \frac{\|a\|^{2d}+r^{2d}M^{2-2d}}{\|x\|^{2d}+r^{2d}} \\
& =& \frac{\bigg(\frac{1}{1+d\|a\|^{2d-2}+d\|a\|^{d-2}|r|}\bigg)^{2d}\|x\|^{2d}+r^{2d}M^{2-2d}}{\|x\|^{2d}+r^{2d}} \\
& \ge & \min\{\bigg(\frac{1}{1+dM^{2d-2}+dM^{d-1}}\bigg)^{2d},M^{2-2d}\},
\end{eqnarray*}
where the equality follows from (\ref{eq:pp}).
 Therefore, there exists $\mu>0$ such that,  for all $(x,r) \in K$,
 \[
\di((x,r),{\rm Fix\,}T_{C,D}) \le \mu \|(x,r)-T_{C,D}(x,r)\|^{\frac{1}{d}}.
\]

Combining these two cases, we see that  $T_{C,D}$ is bounded  H\"{o}lder regular with exponent $\frac{1}{d}$, and so, the sequence $(x^t,r^t)$ generated by the Douglas--Rachford algorithm converges to a point in ${\rm Fix\,}T_{C,D}=\{0_{\mathbb{R}^n}\} \times \mathbb{R}_+$ at least at the order of $t^{\frac{-1}{2(d-1)}}$.

We note that, for  $n=1$, it was shown in \cite{Heinz_0} (by examining the generated DR sequence directly) that the sequence $x^t$ converges to zero at the order $t^{-\frac{1}{d-2}}$ where $d>2$. Note that
$r^t=\|x^t\|^d$. It follows that the actual convergence rate for $(x^t,r^t)$ for this example is $t^{-\frac{1}{d-2}}$ in the case $n=1$. Thus, our convergence rate estimate
for this example is not tight in the case $n=1$. On the other hand, as noted in \cite{Heinz_0}, their analysis is largely limited to the $2$-dimensional case and it is not clear how it can be extended to the higher dimensional setting.
\end{example}

We now examine an even more concrete example involving a subspace and a lower level set of a convex quadratic function in the plane.

\begin{example}[H\"older regularity of the DR operator involving a ball and a tangent line] \label{ex:ball line}
 Consider the following basic convex semi-algebraic sets in $\mathbb{R}^2$:
  \begin{equation*}
  C:=\{x\in\mathbb{R}^2 \mid x_1=0\} \mbox{~and~} D:=\{x\in\mathbb{R}^2 \mid \|x+(1,0)\|^2\leq 1\},
 \end{equation*}
 which have intersection $C\cap D=\{0\}$.
  We now show that  the DR operator $T_{C,D}$ is bounded H\"older regular. Since $C-D=[0,1]\times\mathbb{R}$, by \cite[Cor.~3.9]{bauschke2004finding}, the fixed point set is given by
   $${\rm Fix} \,T_{C,D}=C\cap D+N_{C-D}(0)=(-\infty,0]\times \{0\}.$$
 We therefore have that
   $$\di(x,{\rm Fix}\,T_{C,D})=\begin{cases}
                              \|x\| & x_1>0, \\
                              |x_2| & x_1\leq 0.
                             \end{cases}$$
  Setting $\alpha: =1/\max\{1,\|x-(1,0)\|\}$, a direct computation shows that
   \begin{equation}\label{eq:example3}
   T_{C,D}x:=\left(\frac{I+R_DR_C}{2}\right)x=\left(\alpha-1-\alpha x_1,\alpha x_2\right),\end{equation}
  and thus
   \begin{equation}\label{eq:x-Tx}
     \|x-T_{C,D}x\|^2 = \left((1-\alpha)+x_1(1+\alpha)\right)^2+\left(x_2(1-\alpha)\right)^2.
   \end{equation}

  Now, fix an arbitrary compact set $K$ and let $M>0$ such that $\|x\|\leq M$ for all $x\in K$. For all $x \in K$, there exists $m \in (0,1]$ such that $\alpha =1/\max\{1,\|x-(1,0)\|\} \in [m,1]$ for all $x \in K$. By shrinking $m$ if necessary, we may assume that
  \begin{equation}\label{eq:m}
  \frac{\sqrt{m^2+2m}}{2} \ge M\frac{m^2}{1+m}.
  \end{equation}
 We now distinguish two cases depending on $\alpha$.

  \emph{Case 1 $(\alpha=1)$:} In this case, we have
  $$\|x-(1,0)\|\leq 1 \implies \|x\|^2\leq 2x_1.$$
  In particular, this shows that $x_1\geq 0$. Now \eqref{eq:x-Tx} gives
  $$\|x-T_{C,D}x\| = 2x_1\geq \|x\|^2=\di^2(x,{\rm Fix}\,T_{C,D}).$$

 \noindent \emph{Case 2 $(\alpha<1)$:} Fix $x \in K$. In this case, we show that
    \begin{equation}\label{claim:1}
  \|x-T_{C,D}x\| \ge \frac{m^2}{2(1+m)}\|x\|^3=\frac{m^2}{2(1+m)}\di^3(x,T_{C,D}x).
  \end{equation}
  To do this, we further divide the discussion into two subcases depending on the sign of $x_1$.

  \emph{Subcase I $(x_1>0)$:} In this case, $\di(x,{\rm Fix}\,T_{C,D})=\|x\|$.
  Note that
  \begin{eqnarray*}
  \|x-T_{C,D}x\|^2 &=& \left((1-\alpha)+x_1(1+\alpha)\right)^2+\left(x_2(1-\alpha)\right)^2 \\
  & \ge & \left(x_1(1+\alpha)\right)^2+ \left(x_2(1-\alpha)\right)^2 \\
  & \ge & (m^2+2m)x_1^2+ (1-\alpha)^2 \|x\|^2,
  \end{eqnarray*}
  where the last inequality follows by the fact that $\alpha \ge m$.
  So, the elementary inequality $\sqrt{a^2+b^2} \ge (a+b)/2$ for all $a,b\ge 0$ implies that
  \begin{equation}\label{eq:use1}
  \|x-T_{C,D}x\| \ge \frac{\sqrt{m^2+2m}}{2}x_1 +\frac{1-\alpha}{2}\|x\|.
  \end{equation}
  From the definition of $\alpha$, we see that
  \[
  1-\alpha=\frac{\|x-(1,0)\|-1}{\|x-(1,0)\|} = \frac{x_1^2-2x_1+x_2^2}{\|x-(1,0)\|(\|x-(1,0)\|+1) }.
  \]
  As $m \le \alpha<1$,
  $\|x-(1,0)\| \le \frac{1}{m}$.
  So,
  \[
  1-\alpha \ge \frac{m^2}{1+m} (x_1^2-2x_1+x_2^2)=\frac{m^2}{1+m}\|x\|^2-2 \frac{m^2}{1+m} x_1.
  \]
  Then, by combining with \eqref{eq:use1}, we deduce
   \begin{align*}
    \|x-T_{C,D}x\| &\geq \frac{\sqrt{m^2+2m}}{2}x_1+\frac{1}{2}\|x\|\left(\frac{m^2}{1+m}\|x\|^2-2 \frac{m^2}{1+m} x_1\right) \\
    &= \frac{m^2}{2(1+m)}\|x\|^3 +x_1\left(\frac{\sqrt{m^2+2m}}{2}-\frac{m^2}{1+m}\|x\|\right) \\
    &= \frac{m^2}{2(1+m)}\|x\|^3 +x_1\left(\frac{\sqrt{m^2+2m}}{2}-\frac{m^2}{1+m}M\right).
   \end{align*}
  The claimed equation \eqref{claim:1} now follows from \eqref{eq:m}.

 \emph{Subcase II $(x_1\leq 0)$:} In this case, $\di(x,{\rm Fix}\,T_{C,D})=|x_2|$ and
  \begin{eqnarray*}
  \|x-T_{C,D}x\| &=& \sqrt{\left((1-\alpha)+x_1(1+\alpha)\right)^2+\left(x_2(1-\alpha)\right)^2} \\
  & \ge & (1-\alpha) x_2.
  \end{eqnarray*}
Similar to Subcase~I, we can show that
  \[
  1-\alpha \ge \frac{m^2}{1+m} (x_1^2-2x_1+x_2^2) \ge \frac{m^2}{1+m} x_2^2.
  \]
  where the last inequality follows from $x_1 \le 0$. Thus, (\ref{claim:1}) also follows in this subcase.

 Combining the two cases we have
  $$\di(x,{\rm Fix}\,T_{C,D})\leq \|x-T_{C,D}x\|^{1/3}\quad\forall x\in K.$$
 That is, $T_{C,D}$ is bounded H\"older regular with exponent $\gamma=1/3$. Therefore, for this example, Corollary \ref{cor:DR holder} implies that  the DR algorithm generated a sequence $\{x^t\}$ which converges to $\bar x \in {\rm Fix}T_{C,D}=(-\infty,0]\times \{0\}$ at least with a sublinear convergence rate of $O(\frac{1}{\sqrt[4]{t}})$. Let $x^t=(x_1^t,x_2^t)$ and $\bar x=(\bar x_1,0)$ with $\bar x_1\le 0$.
As $x^{t+1}=T_{C,D}(x^t)$, by passing to the limit in \eqref{eq:example3}, we have $\bar x_1=\bar \alpha-1-\bar \alpha \bar x_1$ where $\bar \alpha=1/\max\{1,|\bar x_1-1|\}$. If $\bar x_1<0$, then $|\bar x_1-1|>1$, and so, $\bar \alpha=1/(1-\bar x_1)$. This implies that $\bar x_1= (\bar \alpha-1)/(1+\bar \alpha)=\bar x_1/(2-\bar x_1)$ and hence,
$\bar x_1=1$ or $\bar x_1=0$ which is impossible. This shows that $\bar x_1=0$, and so,  $\{x^t\}$ converges to $\bar x=(0,0)$ at worst in a sublinear convergence rate $O(\frac{1}{\sqrt[4]{t}})$ regardless the choice of the initial points.

We now illustrate the sublinear convergence rate  by numerical simulation. To do this, we first randomly generated an initial point in $[-100,100]^2$. We then ran the DR algorithm
for this example (starting with the corresponding random starting point) whilst tracking
the value of $\sqrt[4]{t} \, \|x^t-\bar x\|$ and $\frac{-{\log(\|x^t-\bar x\|)}}{\log(t)}$. The experiment was repeated 200 times, and the results plotted in Figure~\ref{fig}.

\begin{figure}[htb]
\begin{center}
 \includegraphics[scale=0.65]{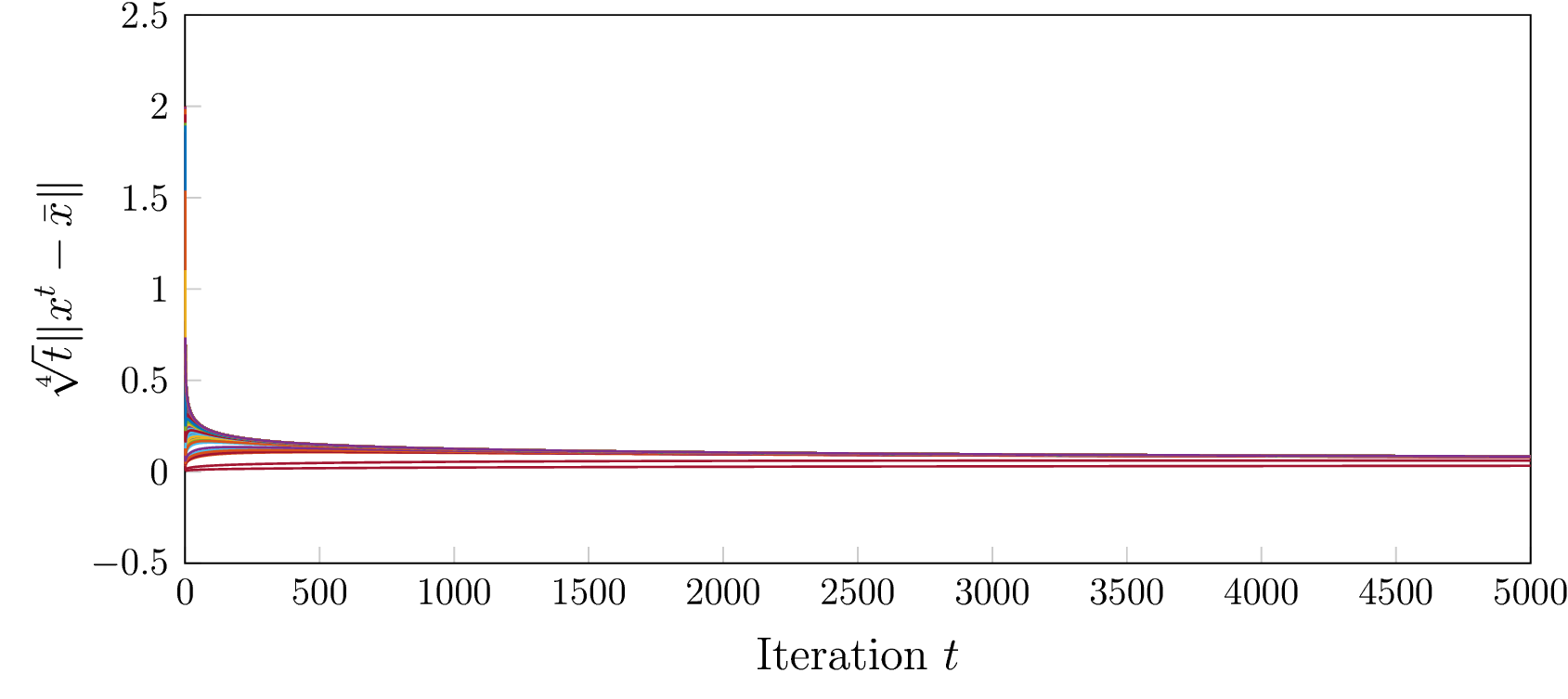}

 \includegraphics[scale=0.65]{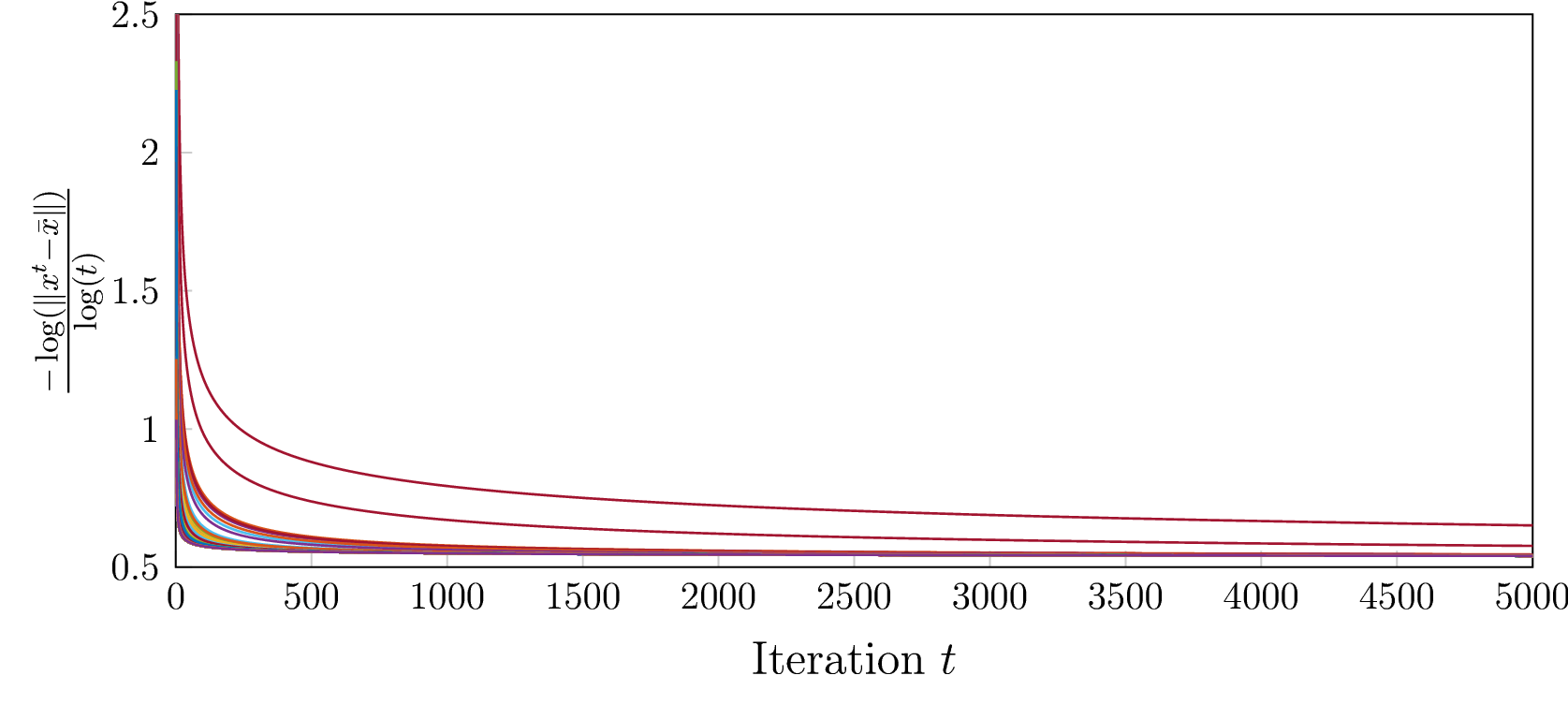}
\end{center}
\caption{Numerical simulation results: (top) the successive change $\sqrt[4]{t} \, \|x^t-\bar x\|$ and (bottom) the ratio $\frac{-{\log(\|x^t-\bar x\|)}}{\log(t)}$ as a function of the number of iterations, $t$.\label{fig}}
\end{figure}

From the first graph, we see that the value of $\sqrt[4]{t} \, \|x^t-\bar x\|$ quickly  decreases with increasing $t$. This supports the result that
$x^t$ converges at least in the order of $O(1/\sqrt[4]{t})$. From the second graph, the value of $\frac{-{\log(\|x^t-\bar x\|)}}{\log(t)}$ appears to approach $1/2$. This suggests that the actual sublinear convergence rate for this example is $O(1/\sqrt{t})$, regardless of the choice of the initial point.
\end{example}

 Furthermore, the following example shows that, whenever the initial point is chosen in the region specified below, the sequence in Example~\ref{ex:ball line} converges with an exact order $O(1/\sqrt{t})$ and thus supports the conjectured rate of convergence.

\begin{example}[The sequence in Example~\ref{ex:ball line} with specific initial points]
 Consider the setting of Example~\ref{ex:ball line}, and suppose that the initial point $x^0=(u_0,v_0)\in\mathbb{R}_{--}\times(0,1)$. If $x^t=(u_t,v_t)\in\mathbb{R}_{--}\times(0,1)$, then using \eqref{eq:example3} we deduce that
  $$x^{t+1}= T_{C,D}(x^t)=\frac{(1-u_t,v_t)}{\sqrt{(1-u_t)^2+v_t^2}}-(1,0)\in\mathbb{R}_{--}\times(0,1).$$
Inductively, the Douglas--Rachford sequence $\{x^t\}$ is contained in $\mathbb{R}_{--}\times\mathbb{R}_{++}$. By Example~\ref{ex:ball line},  the sequence $x^t=(u_t,v_t) \rightarrow (0,0)$.  Below we verify that the sequence with an exact sublinear convergence order $O(1/\sqrt{t})$.

 To see this, we note from $u_t<0$ that
  $$v_{t+1}=\frac{v_t}{\sqrt{(1-u_t)^2+v_t^2}}<\frac{v_t}{\sqrt{1+v_t^2}}.$$
 Setting $w_t = :v_t^2$, we deduce
  $$w_{t+1}<\frac{w_t}{1+w_t}=w_t-w_t^2+O(w_t^3).$$
 Since $w_t\to 0$, for sufficiently large $t$, we have
 $$w_{t+1}<w_t-\frac{1}{2}w_t^2\implies \frac{1}{w_{t+1}}-\frac{1}{w_t}>\frac{1}{2-w_t} \implies \liminf_{t\to\infty}\left(\frac{1}{w_{t+1}}-\frac{1}{w_t}\right)\geq \frac{1}{2}.$$
 It now follows that
  \begin{align*}
   \left(\liminf_{t\to\infty}\frac{1/\sqrt{t}}{v_t}\right)^2=\liminf_{t \to\infty}\frac{1}{t}\frac{1}{w_t} &= \liminf_{t\to\infty}\frac{1}{t}\left(\frac{1}{w_{t}}-\frac{1}{w_0}\right) \\
  &=\liminf_{t\to\infty}\frac{1}{t}\sum_{n=0}^{t-1}\left(\frac{1}{w_{n+1}}-\frac{1}{w_n}\right) \geq \frac{1}{2}.
  \end{align*}
 Taking square roots and inverting both sides we obtain
  \begin{equation}\label{ut}
  \limsup_{t\to\infty}\frac{v_t}{1/\sqrt{t}}\leq \sqrt{2}.\end{equation}
  Now, recall that
  \begin{eqnarray*}
  u_{t+1}= \frac{1-u_t}{\sqrt{(1-u_t)^2+v_t^2}}-1&=& \frac{(1-u_t)-\sqrt{(1-u_t)^2+v_t^2}}{\sqrt{(1-u_t)^2+v_t^2}} \\
  &=& \frac{-v_t^2}{\sqrt{(1-u_t)^2+v_t^2} \big( (1-u_t)+\sqrt{(1-u_t)^2+v_t^2}\big)}.
  \end{eqnarray*}
  Since $\sqrt{(1-u_t)^2+v_t^2} \big( (1-u_t)+\sqrt{(1-u_t)^2+v_t^2}\big) \rightarrow 2$ as $t \rightarrow \infty$, whenever $t$ is sufficiently large we have
  \begin{equation}\label{vt}
  0> u_{t+1} \ge -v_t^2.
  \end{equation}
  Combining \eqref{ut} and \eqref{vt}, we see that there exists $C>0$ such that  $\|(u_t,v_t)\| \le C \frac{1}{\sqrt{t}}$ for all $t\in\NN$. In particular, this also shows that $u_t \rightarrow 0$.

  Noting that $\frac{v_{t+1}}{v_t} =\frac{1}{\sqrt{(1-u_t)^2+v_t^2}}\to 1$ as $t\to\infty$ and $v_t>0$, we therefore deduce that $v_{t-1}<2 v_t$ for all sufficiently large $t$. Combined with \eqref{vt}, this yields
  \[
  v_{t+1}=\frac{v_t}{\sqrt{(1- u_t)^2+v_t^2}} \ge  \frac{v_t}{\sqrt{(1+v_{t-1}^2)^2+v_t^2}} =\frac{v_t}{\sqrt{1+9v_t^2+16^2 v_t^4}} > \frac{v_t}{1+\frac{9}{2}v_t^2}.
  \]
As before, we set $w_t := v_t^2$. Since $w_t\to 0$ and  $(1+\frac{9}{2}w_t)^2(1-10w_t)=1-w_t-\frac{279}{4}w_t^2-\frac{405}{2}w_t^3<1$,   we deduce
  $$w_{t+1}> \frac{w_t}{(1+\frac{9}{2}w_t)^2} > w_t(1-10w_t) \implies w_{t+1}>w_t- 10w_t^2.$$
 Proceeding as before, we obtain
  \[
  \liminf_{t\to\infty}\frac{v_t}{1/\sqrt{t}}\geq \frac{1}{10}.
  \]
 This shows that $\|(u_t,v_t)\| \ge \frac{1}{10}\,\frac{1}{\sqrt{t}}$. Altogether, we have proven that $(u_t,v_t) \rightarrow (0,0)$ with an exact sublinear convergence order $O(1/\sqrt{t})$.
\end{example}

\section{Conclusions}\label{sec:conclusion}
 In this paper, using a H\"older regularity assumption, sublinear and linear convergence of fixed point iterations described by averaged nonexpansive operators has been established. The framework was then specialized to various fixed point algorithms including Krasnoselskii--Mann iterations, the cyclic projection algorithm, and the Douglas--Rachford feasibility algorithm along with some variants. In the case where the underlying sets are convex semi-algebraic, in a finite dimensional space, the results apply without any further regularity assumptions.

  In particular, for our damped Douglas--Rachford algorithm, an explicit estimate for the sublinear convergence rate  has been provided in terms of the dimension and the maximum degree of the polynomials which define the convex sets. We emphasize that, unlike the for damped Douglas--Rachford algorithm, we were not able to provide an explicit estimate of the sublinear convergence rate for the classical Douglas--Rachford algorithm when the two convex sets are described by convex polynomials. Our approach relies on the {\L}ojasiewicz's inequality which gives no quantitative information regarding the H\"older exponent. Providing explicit estimates is left as an open question for future research.

 Another area for future research involves characterization of the convergence rate in the absence of H\"older regularity properties. For instance,
 it is known that the alternating projection method can exhibit arbitrarily slow convergence when applied to two subspaces in infinite dimensional spaces
 without closed sum \cite{HDH}. As shown in \cite[Cor.~3.1]{BorTam},  if only two sets are involved and the initial point is chosen in a specific way,
 the cyclic Douglas--Rachford method can coincide with the alternating projection method, and so, it may exhibit
 arbitrarily slow convergence. On the other hand, it was shown in Proposition~\ref{prop:T_DR hoelder} that the basic/cyclic Douglas--Rachford method enjoys a
 sublinear convergence rate if the underlying sets are convex semi-algebraic sets in finite dimensional spaces. It would be interesting to see
 whether an arbitrarily slow convergence can happen for these two methods for general closed and convex sets in finite dimensional spaces.

 Finally, the current definition of basic semi-algebraic convex sets only applies to finite dimensional spaces. It would interesting to see if a suitable extension of the notion can be profitably used in infinite dimensional spaces using, for instance, polynomials as defined in \cite{GJL}.

\paragraph{Acknowledgments} JMB is supported, in part, by the Australian Research Council. GL is supported, in part, by the Australian Research Council. MKT is supported by Deutsche Forschungsgemeinschaft Research Training Grant 2088. The work was partially performed during his candidature at the University of Newcastle where he was supported by an Australian Postgraduate Award. The authors wish to thank Neal Hermer, Victor Isaac Kolobov, Simeon Reich, Rafa\l\ Zalas and the three anonymous referees for their insightful comments.

\end{document}